\documentclass[12pt]{amsart}

\usepackage{amsfonts,amsmath,latexsym,amssymb,verbatim,amsbsy,times}
\usepackage{amsthm}%, pdfsync}
\usepackage{graphicx}
\usepackage{caption}
\usepackage{subcaption}
\usepackage{dsfont}

\usepackage[top=0.8 in, bottom=0.8in, left=0.8in, right=0.8in]{geometry}

\usepackage[colorlinks=true, pdfstartview=FitV, linkcolor=blue,citecolor=green, urlcolor=blue]{hyperref}

\usepackage{enumitem}

\theoremstyle{plain}
\newtheorem{THEOREM}{Theorem}[section]
\newtheorem{COROL}[THEOREM]{Corollary}
\newtheorem{LEMMA}[THEOREM]{Lemma}
\newtheorem{PROP}[THEOREM]{Proposition}

\theoremstyle{definition}

\theoremstyle{remark}
\newtheorem{REMARK}[THEOREM]{Remark}

\newcommand{\N}{\ensuremath{\mathbb{N}}}   %%% naturals
\newcommand{\R}{\ensuremath{\mathbb{R}}}   %%% reals
\newcommand{\T}{\ensuremath{\mathbb{T}}}   %%% torus

\def \one {{\mathds{1}}}

\def \a {\alpha}
\def \b {\beta}
\def \d {\delta}
\def \g {\gamma}
\def \e {\varepsilon}
\def \f {\varphi}

\def \k {\kappa}
\def \l {\lambda}

\def \n {\nabla}

\def \O {\Omega}

\def \bu {{\bf u}}

\def \bV {{\bf V}}

\def \bX {{\bf X}}

\def \id {{\bf id}}

\def \cA {\mathcal{A}}

\def \cD {\mathcal{D}}

\def \cH {\mathcal{H}}

\def \cM {\mathcal{M}}

\def \cP {\mathcal{P}}

\def \cZ {\mathcal{Z}}

\def\cprime{$'$}

\def \< {\langle}
\def \> {\rangle}
\def \p {\partial}

\def \ss {\subset}

\DeclareMathOperator{\diam}{diam} %
\DeclareMathOperator{\supp}{supp} %
\DeclareMathOperator{\diver}{div} %
\DeclareMathOperator{\tr}{Tr} %

\newcommand{\dist}[2]{\mathrm{dist}(#1,#2)}

\def \cP {\mathcal{P}}

\def \dd  {\mathrm{d}}
\def \dalpha  {\, \mathrm{d}\alpha}
\def \dt  {\, \mathrm{d}t}
\def \dt  {\, \mathrm{d}t}

\def \dx  {\, \mathrm{d}x}
\def \dy  {\, \mathrm{d}y}

\def \ddt  {\frac{\mbox{d\,\,}}{\mbox{d}t}}
\def \dm  {\, \mathrm{d}m}
\def \dbarm  {\, \mathrm{d}\overline{m}}
\def \dr  {\, \mathrm{d}r}

\def \dnu  {\, \mathrm{d}\nu}

\def \dbmu  {\, \mathrm{d}\overline{\mu}}
\def \ds  {\, \mbox{d}s}

\def \barrho {\overline{\rho}}
\def \barm {\overline{m}}
\def \barX {\overline{X}}
\def \barbX {\overline{\mathbf{X}}}
\def \uphi  {\underline{\phi}}

\begin{document}
	
\title[Mass Concentration]{Geometric Structure of Mass Concentration Sets \\ for Pressureless Euler Alignment Systems}

\author{Daniel Lear}
\address{Department of Mathematics, Statistics, and Computer Science, University of Illinois, Chicago}
\email{lear@uic.edu}

\author{Trevor M. Leslie}
\address{Department of Mathematics, University of Wisconsin, Madison}
\email{tleslie2@wisc.edu}

\author{Roman Shvydkoy}
\address{Department of Mathematics, Statistics, and Computer Science, University of Illinois, Chicago}
\email{shvydkoy@uic.edu}

\author{Eitan Tadmor}
\address{Department of Mathematics, Center for Scientific Computation and Mathematical Modeling (CSCAMM), and Institute for Physical Sciences \& Technology (IPST), University of Maryland, College Park}
\email{tadmor@cumd.edu}

\thanks{\textbf{Acknowledgment.} The work of RS was supported in part by NSF	grant DMS-1813351. ET research was supported in part by NSF and ONR grants DMS16-13911 and N00014-1512094. }

\date{\today}

\begin{abstract}
We study the limiting dynamics of the Euler Alignment system with a smooth, heavy-tailed interaction kernel $\phi$ and unidirectional velocity $\mathbf{u} = (u, 0, \ldots, 0)$.  We demonstrate a striking correspondence between the entropy function $e_0 = \partial_1 u_0 + \phi*\rho_0$ and the limiting `concentration set', i.e., the support of the singular part of the limiting density measure. In a typical scenario, a flock experiences \textit{aggregation} toward a union of $C^1$ hypersurfaces: the image of the zero set of $e_0$ under the limiting flow map.  This correspondence also allows us to make statements about the fine properties associated to the limiting dynamics, including a sharp upper bound on the dimension of the concentration set,  depending only on the smoothness of $e_0$.  In order to facilitate and contextualize our analysis of the limiting density measure, we also include an expository discussion of the wellposedness, flocking, and stability of the Euler Alignment system, most of which is new.
\end{abstract}

\maketitle	
%\setcounter{tocdepth}{1}
%\tableofcontents

\vspace{-8mm}

\section{The Euler Alignment System and its Long-Time Dynamics}

We consider the Euler Alignment system on $\R^n$, which is usually written in the following way:

\begin{equation}
\label{e:m}
\left\{ 
\begin{split}
& \p_t\rho(x,t) + \diver(\rho \mathbf{u})(x,t) = 0, \quad \quad x\in \R^n, \\
& \p_t \mathbf{u}(x,t) + \mathbf{u}\cdot \n \mathbf{u}(x,t) = \k\int_{\R^n} \phi(x-y)(\mathbf{u}(y,t) - \mathbf{u}(x,t))\rho(y,t)\dy, \\
& \mathbf{u}(x,0) = \mathbf{u}_0(x); \; \rho(x,0) = \rho_0(x)\ge 0, \;\; \int_{\R^n} \rho_0(x)\dx = M_0<+\infty.
\end{split}\right. 
\end{equation}

Here $\rho$ denotes the density profile, assumed to be compactly supported at time zero, and $\mathbf{u}$ denotes the ($\R^n$-valued) velocity field.  The function $\phi$ represents the (nonnegative) communication protocol, and the parameter $\k>0$ governs the strength of the communications.    

In our analysis, we will make two main assumptions.  First, we will assume that $\phi$ is smooth, radially decreasing, and \textit{heavy-tailed}, i.e., $\int_0^\infty \phi(r)\dr = +\infty$. Second, and most importantly, we will consider velocities that are \textit{unidirectional}; that is, 
\begin{equation}\label{e:uni}
\mathbf{u}(x,t) = (u(x,t), 0, \ldots, 0).
\end{equation}
The utility of these assumptions will be clarified below.  

\subsection{Features of the Euler Alignment System}

\subsubsection{Flocking and Alignment}
The Euler Alignment system provides a hydrodynamic description of the celebrated Cucker--Smale system of ODE's \cite{CS2007a}, \cite{CS2007b}, the salient feature of which is its `flocking' dynamics.  In the hydrodynamic setting, we use the following terminology:
\begin{align}
\label{e:Flocking}
\text{Flocking:} & \quad \sup_{t\ge 0} \supp(\rho(\cdot,t)) = \overline{\cD} <+\infty. \\
\label{e:Alignment}
\text{Alignment:} & \quad \mathbf{u}(\cdot, t)\to \overline{\mathbf{u}} = \text{ const.} \\
\label{e:SFlocking}
\text{Strong Flocking:} & \quad \rho(x-t\overline{\mathbf{u}}, t)\to \rho_\infty(x).
\end{align}
Of course, parsing \eqref{e:Alignment} and \eqref{e:SFlocking} requires specification of the sense in which the convergences take place; the topologies used are context-dependent. There is only one possible candidate for the putative limiting velocity $\overline{\bu}$, namely ratio $\overline{\bu}=\overline{\bu}_0: = \frac{1}{M_0} \int_{\R^n} \rho_0 \bu_0(x)\dx$. The Euler Alignment system is also studied in the periodic setting $x\in \T^n$, where \eqref{e:Alignment} and \eqref{e:SFlocking} are still meaningful but \eqref{e:Flocking} is not. 

The threshold question of whether any of \eqref{e:Flocking}, \eqref{e:Alignment}, \eqref{e:SFlocking} occurs is a well-studied problem, at both the discrete and hydrodynamic levels. Despite the copious effort devoted to the investigation of this phenomenon, much remains to be understood. It is generally difficult to determine whether the agents or trajectories spread out slowly enough that $\phi$ can work to align their velocities (thus decreasing their tendency to spread out) before they escape the regime where $\phi$ is strong enough to do so.  Working in the  context of heavy-tailed kernels eliminates most of these issues: any smooth solution in this case experiences flocking and alignment.  The next stage in studying long time behavior can be focused on understanding the limiting density profile, which will exist in the space of measures even if the density becomes unbounded as $t\to +\infty$.  In this present work we give an exhaustive answer to this question for the class of unidirectional flocks \eqref{e:uni}.

\subsubsection{Wellposedness Considerations and the Quantity $e$}

\label{ss:wpe}

A quantity central to the analysis of \eqref{e:m} is 
\begin{equation}
e(x,t) = \diver \mathbf{u}(x,t) + \phi*\rho(x,t),
\end{equation}
which satisfies the equation 
\begin{equation}
\label{e:nDe}
\p_t e + \diver(e\mathbf{u}) =  (\diver \mathbf{u})^2 - \tr[(\n \mathbf{u})^2].
\end{equation}

The equation \eqref{e:nDe} becomes a conservation law with right hand side zero for unidirectional solutions (and in particular in spatial dimension one).  Equipped with this additional structure, Carrillo, Choi, Tadmor, and Tan proved in \cite{CCTT2016} that when $n=1$, a unique global-in-time solution to \eqref{e:m} exists for sufficiently regular initial data if and only if $e_0\ge 0$ on all of $\R$. This result was extended to unidirectional flows in \cite{LearShvydkoy2019}. In general, however, a sharp critical threshold condition is not known for $n\ge 2$. The work \cite{HT2016} proves global-in-time existence when $e_0\ge 0$, under an additional smallness assumption on the spectral gap of the symmetric strain tensor of $\bu_0$.

Let us turn now to our class of solutions \eqref{e:uni} in question. First, we note that the unidirectionality propagates in time, by the maximum principle. Second, the definition of $e$ and the equation it satisfies become
\begin{equation}
\label{e:euni}
e(x,t) = \p_{x_1} u(x,t) + \phi*\rho(x,t),
\end{equation}
\begin{equation}
\label{e:econsuni}
\p_t e + \p_{x_1}(eu)=0.
\end{equation}
The continuity equation takes a similar form
\begin{equation}
\label{e:rhoconsuni}
\p_t \rho + \p_{x_1}(\rho u)=0.
\end{equation}
Thus, the unidirectional system \eqref{e:euni}--\eqref{e:rhoconsuni} consists of a family of coupled scalar conservation laws.  What stops the unidirectional dynamics from being completely one-dimensional is the convolution term in \eqref{e:euni}, which depends on values of the density in all  stratification layers.  Wellposedness for the system \eqref{e:euni}--\eqref{e:rhoconsuni} for solutions satisfying $e_0\ge 0$ is presented in \cite{LearShvydkoy2019}.

One explanation for the prominent role of $e_0$ in the 1D wellposedness theory is that the quantity
\begin{equation}
\label{e:eint1D}
\int_\a^\b e_0(\g)\dd\gamma
\end{equation}
controls the long-time separation of the trajectories $X(\a,t)$ and $X(\b,t)$ originating at $\a,\b$.  If the quantity \eqref{e:eint1D} is negative, the trajectories intersect at some finite time (which is at most $(\b-\a)/(\k \int_\a^\b e_0(\g)\dd\g)$). If \eqref{e:eint1D} is nonnegative, then the separation is bounded below by a constant multiple of \eqref{e:eint1D}, plus some time-dependent factor that decays to zero as $t\to +\infty$.  In the special case of a heavy-tailed kernel, the long-time separation is also bounded \textit{above} by a constant multiple of \eqref{e:eint1D}.  Thus in the borderline case where $\int_\a^\b e_0(\g)\dd\gamma=0$, the trajectories $X(\a,t)$ and $X(\b,t)$ approach each other asymptotically as $t\to +\infty$, trapping the mass initially inside $[\a,\b]$ in an interval of vanishingly small length.  Thus, if $e_0\equiv 0$ on an interval of nonzero mass, we observe a {\em mass concentration phenomenon}, which manifests itself in  the emergence of Dirac atoms in the limiting measure $\barm$.  The relationship between $e_0$ and the spread of the limiting trajectories will be central to our analysis below.

The last observation was first quantified in the form above by the second author in the recent paper \cite{L2019CTC}, which analyzed the compatibility of the condition $e_0\ge 0$ with restriction of the domain to the non-vacuum region.  The analysis of \cite{L2019CTC} was partly inspired in turn by the work  \cite{Tan2019WeakSing} by Tan, who showed that, in the case of weakly singular kernels (i.e., $\phi(x)\sim |x|^{-s}$ near $x=0$, with $s\in (0,1)$), an interval of positive mass on which $e_0\equiv 0$ can collapse to a point in finite time (unlike the situation for smooth kernels).  

\begin{REMARK}
It is instructive to consider the Euler alignment system in special case of a constant kernel, $\phi\equiv 1$ (the strength of the interactions being encoded in the constant $\k$, which we allow to take the value zero in this remark). In this case the unidirectional \eqref{e:m} reads
\[
\left\{ 
\begin{split}
& \p_t\rho(x,t) + \p_{x_1}(\rho u)(x,t) = 0, \\
& \p_t u(x,t) + u\p_{x_1}{u}(x,t) = \k M_0 \big(\overline{u}_0 - u(x,t)\big), 
\end{split}\right. \qquad x=(x_1,x_-)\in \R^n,
\]
subject to prescribed initial data ${u}(x,0) = {u}_0(x); \; \rho(x,0) = \rho_0(x)\ge 0$.
Here $\overline{u}_0$ is the average velocity $\displaystyle \overline{u}_0:= \frac{1}{M_0}\int \rho_0 u_0(x)\dd x$.  We distinguish between three different regimes, depending on the initial configuration of $(\rho_0,u_0)$.
In case of \emph{sub-critical} initial data, $e_0= \p_{x_1}u_0 +\k M_0>0$ the system admits global smooth solutions.
In case of \emph{super-critical} initial data, $e_0= \p_{x_1}u_0 +\k M_0<0$, the system admits finite-time breakdown where $\lim_{\substack{\\ x\rightarrow x_c\\ t\uparrow t_c}}\p_{x_1}u(x,t) = -\infty$ \emph{and} (provided the breakdown occurs along a trajectory where the density is initially positive) there is mass concentration at that point $\lim_{\substack{\\ x\rightarrow x_c\\ t\uparrow t_c}}\p_{x_1}\rho(x,t)=\infty$, 
leading to the formation of \emph{delta shocks} \cite{chen2004concentration,shelkovich2009transport,nilsson2011mass}.  Thereafter, one interprets the unidirectional pressureless system in a weak formulation, 
\[
\left\{ 
\begin{split}
& \p_t\rho(x,t) + \p_{x_1}(\rho u)(x,t) = 0, \\
& \p_t (\rho u)(x,t) + \p_{x_1}(\rho u^2)(x,t) = \k M_0 \rho(x,t)\big(\overline{u}_0 - u(x,t)\big), 
\end{split}\right. \qquad x=(x_1,x_-)\in \R^n.
\]
Entropic solutions of pressureless equations with super-critical data, at least in the 1D case, is the subject of extensive studies,  realized in a variety of different approaches and we mention two---variational formulations \cite{tadmor2011variational,liu2016least, HaHuangWang2014wk} and sticky particles \cite{brenier1998sticky,weinan1996generalized,nguyen2008pressureless}. Of these, only \cite{HaHuangWang2014wk} treats the case where $\k>0$.  

The current paper covers the third regime---a borderline case  with critical initial configurations such that $e_0\geq 0$. The zero-level set of $e_0$ then leads to mass concentration at $t=\infty$. The fascinating aspect, to be made precise  in Theorems~\ref{t:Xreg} and  \ref{t:fineprops} below, is the way in which the geometry of the singular part of the limiting mass measure depends on the zero set of $e_0$ and the regularity of $e_0$ near its zero set.  This motivates  further study for the geometry of delta-shocks in super-critical cases $e_0<0$.
\end{REMARK}

\subsubsection{Fast Alignment, Strong Flocking, and the Limiting Trajectory Map}

Let us consider the particle flow map generated by the field $\mathbf{u}$: 
\[
\dot{\bX}(\a,t) = \mathbf{u}(\bX(\a,t), t),
\quad \quad \bX(\a,0) = \a, \quad \quad \a\in \R^n.
\]
Although the flow-map is defined globally in $\R^n$, we often require a compact convex domain of labels $ \a \in \O \ss \R^n$ to study long time behavior. Such domains are always assumed to contain the material flock
\begin{equation}
	\supp \rho_0 \ss \O, \quad \supp \rho(t) \ss \O(t):=\bX(\O,t).
\end{equation}

On any such compact domain, a global solution to \eqref{e:m} experiences flocking, see \cite{TT2014}, 
\begin{align}
\cA(t) =|\mathbf{u}(\cdot, t) - \overline{\mathbf{u}}_0|_{L^\infty(\O)} & \le \cA_0 e^{-\d_\O t}, \label{e:ff} \\
\sup_{t \geq 0} \diam \O(t) & < \infty.
\end{align}
By Galilean invariance of the system \eqref{e:m}, we may assume without loss of generality that $\overline{\bu} = 0$.  
Then the exponential decay of $|\mathbf{u}|$ implies that the particle trajectory map converges uniformly on any compact $\O$, to a continuous function: 
\[
\|\bX(\cdot,t)  -  \barbX\|_{L^\infty(\O)} \leq Ce^{-\d_\O t}, \quad t>0.
\]
As long as $e_0\ge c >0$  on the initial flock $\supp \rho_0$, the alignment estimate \eqref{e:ff} can be lifted to higher regularity classes, effectively showing exponential flattening of velocity field, and convergence of density to a smooth traveling wave, see \cite{ShvydkoyTadmorII}, \cite{LearShvydkoy2019}. 
In fact, those arguments produce local flattening  even without the uniform positive lower bound on $e_0$: assuming $e_0\ge 0$ everywhere and denoting
\[
\cP_\e = \{\a\in \R^n: e_0(\a)\ge \e\}, \quad \cP = \{\a\in \R^n: e_0(\a) > 0\}.
\]
one has
\begin{equation}
\label{e:fastalignment}
\sup_{\a\in \cP_\e \cap \O} |\n u(\bX(\a,t),t)| %+\sup_{\a\in \cP_\e} |\n^2 u(\bX(\a,t),t)|
\le Ce^{-\d_{\e,\O} t}.
\end{equation}
Solving the continuity equation
\[
\rho(\bX(\a,t), t) = \rho_0(\a)\exp\left\{ - \int_0^t \p_{x_1} u(\bX(\a,s),s)  \ds \right\},
\]
one can observe a local strong flocking along these same trajectories:
\begin{align*}
\sup_{\a\in \cP_\e \cap \O} |\rho(\bX(\a,t), t) - f(\a))| = Ce^{-\d_{\e,\O} t},
\end{align*}
for some smooth limiting function $f$, defined on $\cP\cap \O$. We can thus focus our attention on the complementary zero-set of $e_0$:
\[
\cZ = \{\a\in \R^n: e_0(\a) = 0\}.
\]
This is where the mass concentration phenomenon we discussed earlier presents itself. We expect that  the density will aggregate on the Lebesgue-negligible set $\barbX(\cZ)$ in the sense of weak-$*$ convergence of measures. To study this concentration phenomenon we introduce the limiting density-measure as follows. Denote  $\dm_t(x):=\rho(x,t)\dx$. According to the continuity equation this is a push-forward of the initial mass by the Lagrangian flow-map
\[
m_t = \bX(\cdot , t)_\sharp m_0,\quad m_t(E) = m_0(\bX^{-1}(E,t)).
\]
The measures $m_t$ converge weakly-$*$ to the push-forward of the limiting flow-map: $\barm = \barbX_\sharp m_0$. Indeed, for any $\eta\in C_c(\R^n)$ we have
\[
\int \eta(x) \dbarm(x) \stackrel{t\to +\infty}{\longleftarrow}\int \eta(x) \dm_t(x) 
= \int \eta(\bX(\a,t)) \dm_0(\a)
\stackrel{t\to +\infty}{\longrightarrow}\int \eta( \barbX(\a))\dm_0(\a).
\]
Our main question, then, concerns the structure of the limiting measure $\barm$.  We expect that $\barm$ has a singular part concentrated on $\barbX(\cZ)$, and that the absolutely continuous part $\barrho(x)\dx$ satisfies $\barrho\circ \barbX = f$, where $f$ is as in the previous paragraph.  Theorem \ref{t:main} below formalizes these expectations.  

The discussion above assumes that $\dm_0(x) = \rho_0(x)\dx$ is absolutely continuous with respect to Lebesgue measure; however, there are almost no additional technicalities necessary to include a possibly singular part, so we will do so below.  Allowing this more general class of initial data has the satisfying consequence that entire evolution $m_t$ and its limit $\overline{m}$ belong to the same class $\cM_+(\R^n)$. Furthermore, this viewpoint is true to the kinetic formulation from which the macroscopic version \eqref{e:m} is derived, and the discrete Cucker-Smale system can be viewed as a particular case of purely atomic solutions $m_t = \sum_i m_i \d_{x_i(t)}$.  We clarify in Section \ref{s:msolns} the wellposedness  theory for solutions starting from such initial data.

\subsection{Statement of Results}

The main technical properties of the limiting flow map $\barbX$ that are needed to analyze the structure of $\barm$ are contained in the following Proposition, which is of independent interest. We include it in this section in order to motivate the statement of Theorem \ref{t:main}. Here and below, we write $\a = (\a_1, \a_-)\in \R\times \R^{n-1}$, and we use $\barX$ to denote the first component of $\barbX$; that is, 
\[
\barbX(\a) = (\barX(\a), \a_-).
\]
We will use the notation $|E|$ to denote the $k$-dimensional Lebesgue measure for a subset $E$ of $\R^k$ (the relevant $k$'s being $k=1, n-1, n$).  

\begin{PROP}
\label{p:barbX}
We have the following estimate in the $x_1$ direction:
\begin{equation}
\label{e:Xbarsep}
\frac{1}{\k M_0 \|\phi\|_{L^\infty}} \int_{\b}^{\g} e_0(\zeta, \a_-)\dd\zeta
\le \barX(\g, \a_-)-\barX(\b, \a_-)
\le \frac{1}{\k M_0 \uphi} \int_{\b}^{\g} e_0(\zeta, \a_-)\dd\zeta.
\end{equation}
The lower bound is valid for any pair $(\b,\a_-)$, $(\g, \a_-)$ of elements of $\R^n$ such that $\b<\g$; the upper bound is valid for such pairs that belong to $\O$.  Consequently, $|\barbX(\cZ)|=0$.
\end{PROP}

Our first main Theorem shows that the two sets $\cP$ and $\cZ$ are in natural correspondence with the pieces of the Lebesgue decomposition of $\barm$.  

\begin{THEOREM}[Structure of $\barm$]
\label{t:main}
Let $\bu(x,t)\in C^1([0,\infty), C^k(\R^n))$, $\dm_t(x)\in C_{w^*}([0,\infty);\cM_+(\R^n))$ be a measure-valued solution to \eqref{e:m}, corresponding to the initial velocity $\bu_0 = (u_0, 0)$, and initial density measure $m_0$.  Let $\dm_0(x) = \rho_0(x)\dx + \dnu$ denote the Lebesgue decomposition of $m_0$ with respect to Lebesgue measure, and let $\barm$ denote the limiting measure: $m_t\stackrel{*}{\rightharpoonup} \barm$ in $\cM_+(\R^n)$.  Then the Lebesgue decomposition of $\barm$ with respect to Lebesgue measure is determined from the following:
\begin{align}
\dbarm & = \barrho \dx + \dbmu \\
\barrho \dx = \barbX_\sharp (\rho_0 \one_{\cP} \dx),
& \quad \dbmu = 
\barbX_\sharp (\rho_0 \one_{\cZ} \dx + \dnu).\label{e:barmdecomp1}
\end{align}
Here the singular part is supported on  $\barbX(\cZ \cup \supp \nu)$, while the density $\barrho \in L^1$ is supported on $\barbX(\cP \cap \supp \rho_0))$ and is given by  
\begin{equation}
\label{e:barmdecomp2}
 \barrho\circ \barbX =\frac{\rho_0 }{\p_{\a_1} \barX}  \one_{\cP}.
\end{equation}

Consequently, the measure $\barm$ is absolutely continuous with respect to Lebesgue measure if and only if $m_0(\cZ) = 0$ and $\nu = 0$. Finally, we have $\rho(\bX(\cdot,t), t)\to \barrho\circ \barbX$ uniformly on compact subsets of $\cP$.  
\end{THEOREM}

\begin{REMARK}
If $\supp \rho_0 \ss \cP$, then we simply have the global convergence $\| \rho(t) - \bar{\rho}\|_{L^\infty} \to 0$, which recovers the result of \cite{LearShvydkoy2019, ShvydkoyTadmorII}.
\end{REMARK}

\begin{REMARK}[Comparison with \cite{LS2019}]
Combining \eqref{e:barmdecomp2} with (a limiting version of) \eqref{e:Xbarsep} yields the two-sided bound
\begin{equation}
\label{e:rhobarbd}
\k M_0 \uphi \cdot \frac{\rho_0}{e_0}(\a) \le \barrho\circ \barbX(\a) \le \k M_0 \|\phi\|_{L^\infty} \cdot \frac{\rho_0}{e_0}(\a),
\quad \quad \a\in \cP.
\end{equation}		
	
This bound offers a supplement to the conclusions obtained in \cite{LS2019} by the second and third authors. This work treated---in the 1D periodic setting without vacuum---the deviation of $\barrho$ from a constant (in the case where $\cZ=\emptyset$).  The result obtained there, written in the notation of our current context, reads
\[
\limsup_{t\to +\infty} \left\| \rho(t) - \frac{1}{|\T|} \int_\T \rho_0 \right\|_{L^1} \lesssim \left\| \frac{e_0}{\rho_0} - \int \phi \right\|_{L^\infty}.
\]
The significance of the number $\int \phi$ is that periodicity guarantees it to be the average value of $e_0/\rho_0$ on $\T$.  So the result of \cite{LS2019} says that the long-time deviation of $\rho$ from its average value (measured in $L^1$) is controlled by the deviation of the initial quantity $e_0/\rho_0$ from its average value (measured in $L^\infty$).  In the case of global kernels, $\supp \phi = \T$, the bound \eqref{e:rhobarbd} provides an improvement, since it is a two-sided \textit{uniform} bound if $e_0/\rho_0$ is close to its average value on $\T$.  However, the analysis of \cite{LS2019} is valid even for \textit{local} kernels, where $\supp \phi$ is much smaller than $\T$ (in which case the left side of \eqref{e:rhobarbd} vanishes), and for a class of `topological' kernels $\phi$ introduced in \cite{STtopo}, where the communication protocol $\phi$ itself depends on the density (and is in general nonintegrable).  Therefore, while the bound \eqref{e:rhobarbd} provides a nice supplement to the results of \cite{LS2019}, the present work is otherwise essentially disjoint from \cite{LS2019}.
\end{REMARK}

\begin{REMARK}
We return briefly to the consideration of the case $\phi \equiv 1$ for comparison.  Let us drop the assumption of unidirectionality for a moment and consider the flow map $(\bX, \bV)$ associated to a solution $(\bu, \rho)$ that is known to exist globally in time.  In this case one has (assuming momentum zero, for simplicity)
\[
\dot{\bV} = -\k M_0 \bV,
\qquad \dot{\bX} = \bV.
\]
whence
\begin{comment}
\[
\bV(\a,t) = \bu_0(\a) e^{-\k M_0 t}, 
\qquad  \bX(\a,t) = \a + \frac{\bu_0(\a)}{\k M_0} (1 - e^{-\k M_0 t}), 
\]
so that in particular, 
\end{comment}
\[
\barbX(\a) = \a + \frac{\bu_0(\a)}{\k M_0},
\]
upon solving the particle trajectory equations and taking $t\to +\infty$.  Notice that $\n \barbX = \id + \frac{1}{\k M_0} \n \bu_0$, or in the unidirectional case, $\p_{1} \barX = \frac{1}{\k M_0} (\k M_0 + \p_{1} u_0) = \frac{e_0}{\k M_0}$, in agreement with \eqref{e:Xbarsep}.  However, it should be noted that a different argument is needed in order to obtain \eqref{e:Xbarsep} for the case of general $\phi$, where the equation is genuinely nonlocal (unlike for constant $\phi$).  In fact, the argument leading to \eqref{e:Xbarsep} does not extend to the case of non-unidirectional data.  
\end{REMARK}

Later in Section~\ref{s:stab}, we show that the assignment of limiting measure $m_0 \to \barm$ is stable in the Kantorovich-Rubinstein metric. Our argument is a borderline $L^\infty$-version of the $\ell_{p,q}$-stability estimates presented in \cite{Ha-stability}.

Our second Theorem implies that if $\cZ$ is a `nice' set, then the set on which $\dbmu_\cZ :=\barbX_\sharp( \rho_0 \one_{\cZ} \dx)$ concentrates is a union of $C^1$ hypersurfaces.  This relies on some additional regularity of $\barbX$ inside $\cZ$:

\begin{PROP}
\label{p:Xreg}
The map $\barbX$ is continuously differentiable on the complement of $\p \cZ$.	
\end{PROP}  

\begin{THEOREM}
\label{t:Xreg}
Assume $U$ is an open subset of $\cZ$, having the following properties:
\begin{itemize}
\item $U$ is `$x_1$-convex', i.e., if $(\b,\a_-), (\g,\a_-)\in U$, then $((1-\l)\b + \l \g, \a_-)\in U$ for all $\l\in [0,1]$.
\item $U$ contains the graph of a $C^1$ function $f:U_-\to \R$, where $U_-:=\{\a_-\in \R^{n-1}:(\a_1, \a_-)\in U \text{ for some } \a_1\in \R\}$ denotes the projection of $U$ onto $\R^{n-1}$.  That is, assume 
\[
\Gamma(f) = \{(f(\a_-),\a_-): \a_-\in U_-\} \subset U.
\]  
\end{itemize}
Then $\barbX(U)$ is the graph of a $C^1$ function:
\[
\barbX(U) = \barbX(\Gamma(f)) = \{(\barX(f(\a_-), \a_-), \a_-):\a_-\in U_-\}.
\]
In particular, if $\overline{U}$ is all of $\cZ$, then 
\begin{equation}
\dbmu_\cZ(x_1, \a_-) = c(\a_-)\d_{f(\a_-)}(x_1) \dalpha_-,
\end{equation}
where $\dbmu_\cZ$ denotes the pushforward measure $\barbX_\sharp( \rho_0 \one_{\cZ} \dx)$, and 
\[
c(\a_-):=\int_{\b(\a_-)}^{\g(\a_-)} \rho_0(\a_1, \a_-)\dalpha_1,
\quad \quad 
[\b(\a_-),\g(\a_-)] = \{\alpha_1\in \R: (\a_1, \a_-)\in \cZ\}.
\]
The functions $c(\a_-)$ are $C^1$ if $\p U$ is $C^1$.
\end{THEOREM}

Theorem \ref{t:Xreg} says that the solution experiences \textit{aggregation} along horizontal slices in $\cZ$, in a regular way with respect to the other directions.  This is instructive to demonstrate graphically; see the figures below.  Note that while the Corollary is stated for the case of a single set $U$ satisfying the hypotheses, one can combine multiple open sets satisfying the two bullet points to obtain different `sets of aggregation' consisting smooth hypersurfaces, as shown in these figures. In each of the two-dimensional examples below, observe that whenever two curves in the image meet at a point, they must be tangent at that point.  This is because both of the curves must be $C^1$ and cannot cross each other.

\begin{figure}[ht]
	\begin{subfigure}{.43\textwidth}
		\centering
		\includegraphics[width=.8\linewidth]{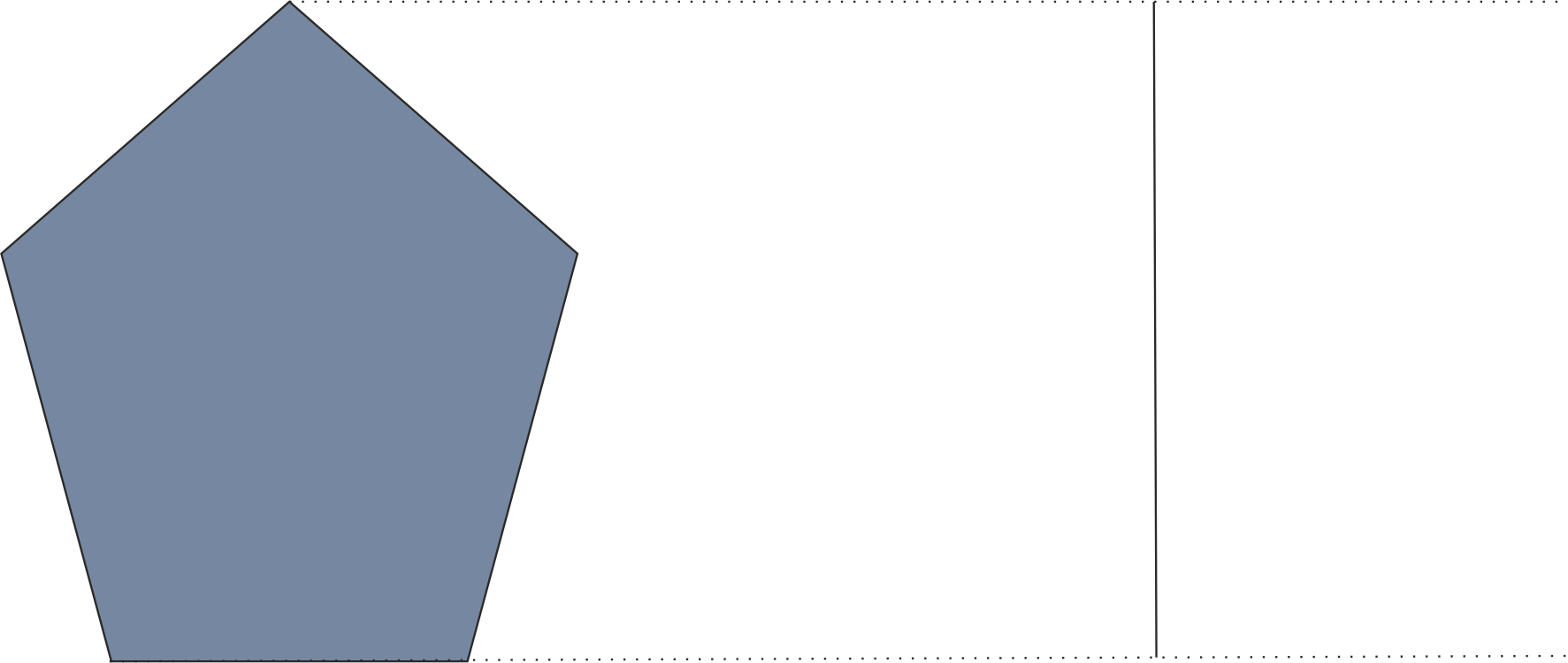}  
		\caption{A domain $\cZ$ which is convex in the $x_1$-direction maps to a smooth curve (depicted here as a line segment for simplicity).}
		\label{fig:pentagon}
	\end{subfigure}
\hfill
	\begin{subfigure}{.55\textwidth}
		\centering
\vspace{3.5 mm}
		\includegraphics[width=.8\linewidth]{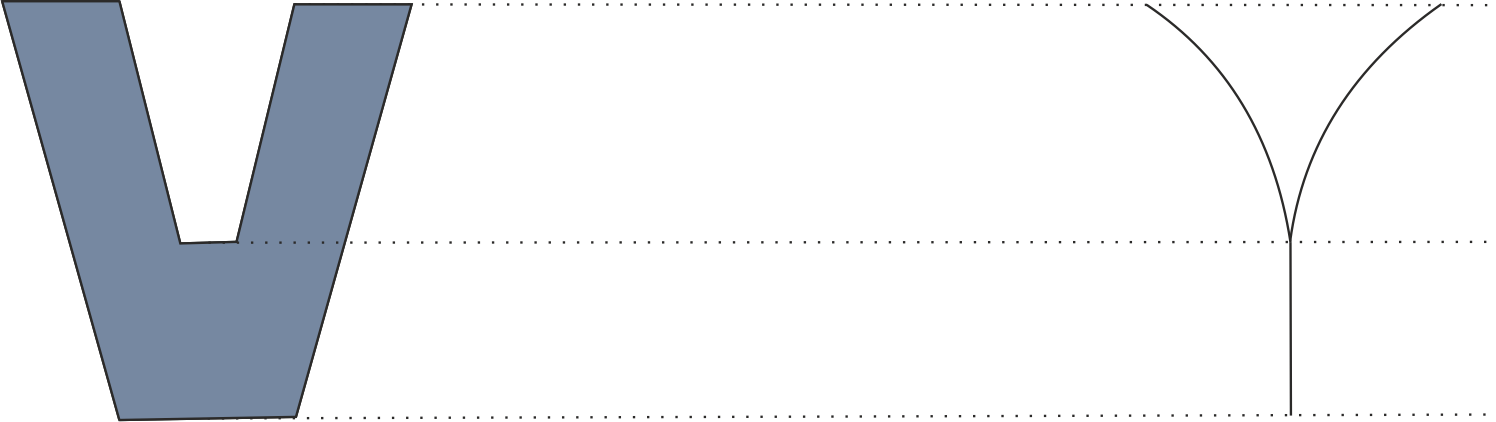}  
		\caption{This set $\cZ$ is not convex in the $x_1$-direction, but it can be decomposed into two (maximal) such domains, whose overlap determines the corresponding overlap of the $\barbX(\cZ)$ curves.}
		\label{fig:V}
	\end{subfigure}
\begin{subfigure}{.4\textwidth}
	\centering
	\includegraphics[width=.8\linewidth]{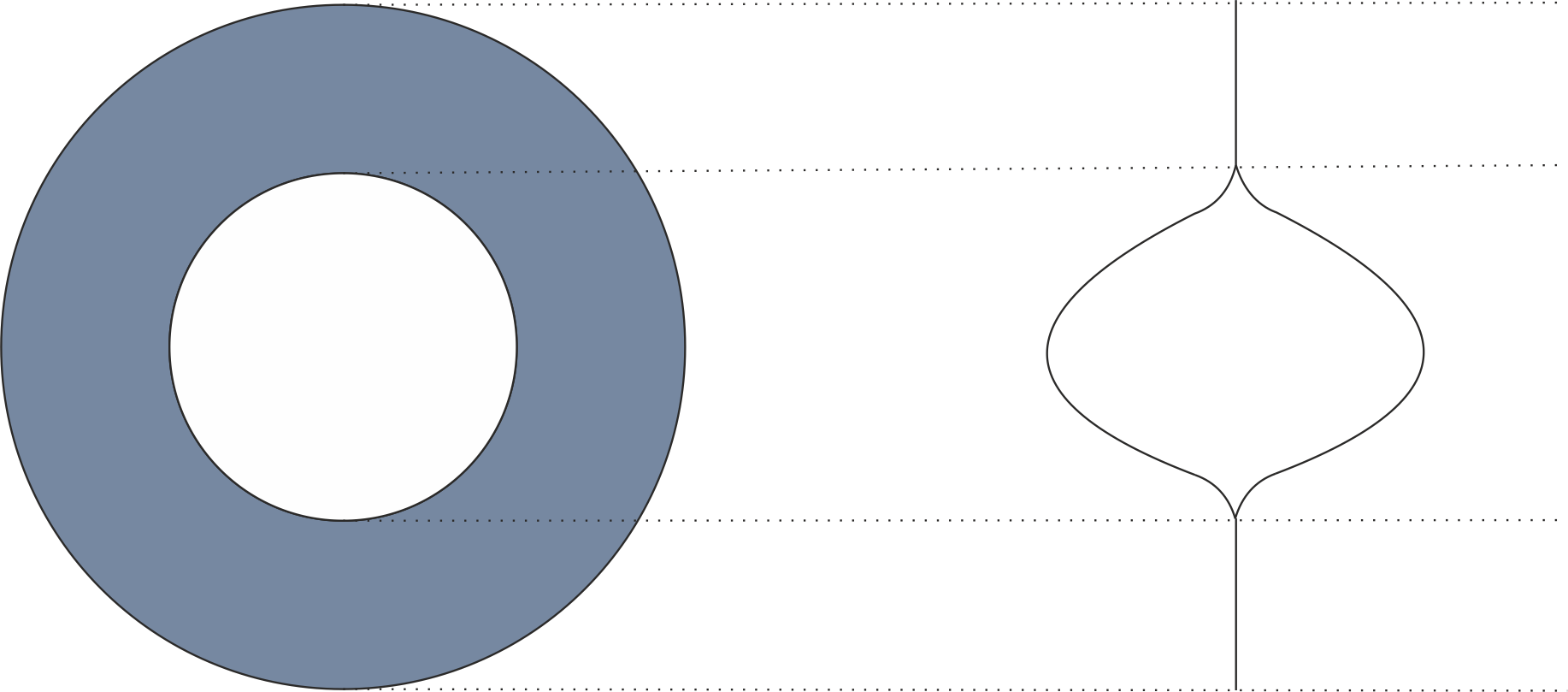}  
	\caption{As in (B), one decomposes the annulus $\cZ$ into maximal $x_1$-convex components in order to determine the structure of $\barbX(\cZ)$.}
	\label{fig:annulus}
\end{subfigure}
\begin{subfigure}{.58\textwidth}
	\centering
	% include second image
	\includegraphics[width=.8\linewidth]{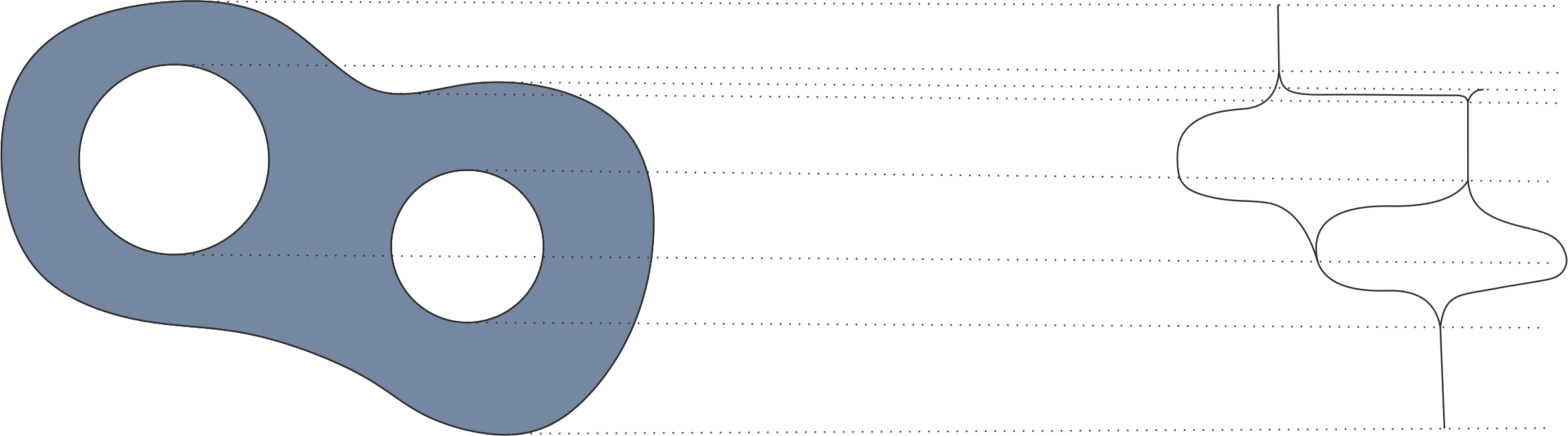}  
	\caption{More complicated sets $\barbX(\cZ)$ can be produced by, for example, increasing the genus of $\cZ$.}
	\label{fig:genus2}
\end{subfigure}
	\caption{A heuristic illustration of the effect of $\barbX$ on $\cZ$. In (B)--(D), the curves comprising $\barbX(\cZ)$ are tangent at each bifurcation point. }
	\label{fig:fig}
\end{figure}

A natural question is to look further into finer properties of the null set $\barbX(\cZ)$, and ask how small or large this set can be in terms of fractal  dimension. We answer this question in 1D (to simplify technical details).  The main result states that the size of $\barX(\cZ)$ is directly tied to the regularity of $e_0$.

\begin{THEOREM}
\label{t:fineprops}
If $n = 1$ and $e_0\in C^k(\R)$, then the upper box-counting dimension of $\barX(\cZ)$ satisfies
\begin{equation}
\label{e:boxdime0}
\overline{\dim}_{\mathrm{box}}(\barX(\cZ)) 
\le \frac{\overline{\dim}_{\mathrm{box}}(\cZ)}{k+1}. 
\end{equation}		
In particular, if $e_0\in C^\infty(\R)$, then the Hausdorff and box-counting dimensions of $\barX(\cZ)$ are both zero. 
\[
\dim_{\mathrm{box}}(\barX(\cZ))  = \dim_\cH(\barX(\cZ)) =  0.
\]
The bound \eqref{e:boxdime0} is sharp in the following sense: For any $k\in \N\cup \{0\}$ and any $\e>0$, there exists initial data such that $e_0\in C^k(\R)$ and $\cZ$ has positive Lebesgue measure, but 
\[
\dim_{\mathrm{box}}(\barX(\cZ)) = \dim_\cH(\barX(\cZ))>\frac{1}{k+1} - \e.
\]
\end{THEOREM}

\begin{REMARK}
We state \eqref{e:boxdime0} in terms of $\overline{\dim}_{\mathrm{box}}$ rather than $\dim_\cH$, because in the `typical' scenario where $\dim_\cH(\cZ) = 1$, our version gives a stronger statement than the corresponding one for Hausdorff dimension. (Recall that the box-counting dimension dominates the Hausdorff dimension, c.f. \eqref{e:dimensions} below.)
\end{REMARK}

\subsection{Outline}

The remainder of the paper is structured as follows.  In Section \ref{s:msolns}, we discuss existence and uniqueness of measure-valued solutions with $C^1_t C^k_x$ velocities, $k\ge 1$.  We recover the expected global-in-time existence and uniqueness for unidirectional data (with $e_0\ge 0$) as a byproduct of our proof of the key estimate \eqref{e:Xbarsep}.  We close Section \ref{s:msolns} with an analysis of flocking and stability.  

With the theory above in hand, we finish the proof of Proposition \ref{p:barbX} in Section \ref{s:barbX} and use it to prove Theorem \ref{t:main}.  All that is needed here is an understanding of $\barbX$ along horizontal slices. On the other hand, lateral regularity is needed in order to prove Theorem \ref{t:Xreg}; in Section \ref{s:Xreg}, we prove Proposition \ref{p:Xreg} (using some of the previously established flocking estimates) before finishing Theorem \ref{t:Xreg}. 

In Section \ref{s:fine}, we study some fine properties of $\barm$ and $\barX$ in one spatial dimension.  The sharpness statement in Theorem \ref{t:fineprops} is proven by building $e_0$ from a certain Cantor set $\cZ$ of positive measure, using Frostman's Lemma together with \eqref{e:Xbarsep} to establish the dimension of $\barX(\cZ)$.  As the proof shows, the dimension depends not only on $\cZ$, but also on the way $e_0$ approaches zero near $\cZ$.  We close on a related note: Starting from a regular density profile $\rho_0$, we can adjust the rate $e_0$ approaches zero at an isolated point in $\cZ$ in order to ensure a specified local dimension of $\barm$ at the corresponding point of $\barX(\cZ)$. 

\section{Wellposedness for  Measure-Valued Solutions}
\label{s:msolns}

In this section, we treat the well-posedness of the system \eqref{e:m} for measure-valued densities. To be explicit, given an initial measure $m_0\in \cM_+(\R^n)$ and an initial velocity $\bu_0\in C^k(\R^n)$, we will discuss the existence, uniqueness, flocking properties, and stability of solutions $\bu\in C^1([0,T), C^k(\R^n))$, $m\in C_{w^{*}}([0,T);\cM_{+}(\R^n))$ to the following system:
\begin{equation}
\label{e:mmeas}
\left\{ 
\begin{split}
& \frac{D\bu}{Dt}(x,t) = \k\int_{\R^n} \phi(x-y)(\bu(y,t) - \bu(x,t))\dm_t(y), 
\qquad  \bu(x,0) = \bu_0(x);\\
& \int_{\R^n} \xi(y,t) \dm_t(y) - \int_{\R^n} \xi(y,0) \dm_0(y) = \int_0^t \int_{\R^n} \frac{D\xi}{Ds}(y,s) \dm_s(y),
\quad \xi\in C^\infty_c(\R^n\times \R). 
\end{split}\right. 
\end{equation}	
Here $\frac{D}{Ds}=\p_s + \bu\cdot \n$ denotes the advective derivative, and $\cM_{+}(\R^n)$ denotes the set of non-negative Radon measures on $\R^n$, endowed with the topology of weak convergence.  By the latter, we mean convergence on the space $C_{b}(\R^n)$ of continuous bounded functions. In what follows, we will say that $\mu_n$ converges weakly to $\mu$ on $\cM_{+}(\R^n)$ and we write $\mu_n \overset{\ast}{\rightharpoonup} \mu$ if
\[
\int_{\R^n} f \mbox{d}\mu_n \rightarrow \int_{\R^n} f \mbox{d}\mu \qquad \text{for all} \quad  f\in C_b(\R^n).
\]
Since we will deal with measures supported on a bounded set $\Omega$, this convergence coincides with the classical weak-$*$ convergence on $C_{0}(\R^n)$, the predual of $\cM(\R^n)$.

\subsection{Lagrangian Formulation and Local Wellposedness}

\label{ss:lwp}

The Euler Alignment system  for Lagrangian velocities $\bV(\cdot,t)=\bu(\bX(\cdot,t),t)$ takes the form
\begin{equation}
\label{e:Lag}
\left\{ 
\begin{split}
& \dot{\bX}(\a,t) = \bV(\a,t),  \\
& \dot{\bV}(\a,t) = -\k \int_{\O} \phi(\bX(\a,t)-\bX(\g,t))[\bV(\a,t) - \bV(\g,t)]\dm_0(\g), \\
& \bX(\a,0) = \a, \; \bV(\a,0) = \bu_0(\a), \; \a\in \O.
\end{split}
\right.
\end{equation} 
Here $\O$ can be taken to be any compact set containing the support of $m_0$ (we also assume convexity for simplicity). Proving the existence of particle trajectories $(\bX, \bV)\in C^1(\R_+;C^k(\O)\times C^k(\O))$ amounts to a routine application of the Picard Theorem, together with a straightforward estimate eliminating finite-time blowup for $\|\bX(t)\|_{C^k(\O)} + \|\bV(t)\|_{C^k(\O)}$.  To ensure that the particle trajectories yield a solution 
\[
(\bu(\cdot,t), \dm_t)=(\bV(\bX^{-1}(\cdot,t),t), \bX(\cdot,t)_\sharp \dm_0)
\] 
to the Eulerian formulation, we need  $\bX(\cdot,t):\O\to \O(t)$ to remain invertible, and we need $\det \n \bX \ne 0$ on $\O$ to ensure $\bu$ remains in $C^k$. This of course holds at least for a short period of time since one starts from $\bX(\cdot, 0) = \id$.   A continuation of solution can be achieved under the following condition: there exists $\e>0$ such that on  time interval $[0,T)$ one has
\begin{equation}
\label{e:Oe}
\inf_{\substack{\a\ne \b} \in \O} \frac{|\bX(\b,t) - \bX(\a,t)|}{|\b-\a|}>\e, 
\qquad t\in [0,T).
\end{equation}
In fact, this implies $|\det \n \bX(\cdot,t)| \geq \e^n$, as \eqref{e:Oe} guarantees that every eigenvalue of $\n \bX(\cdot, t)$ has an absolute value of at least $\e$. 
\begin{THEOREM}
\label{t:exist}	
For any initial data $(\bu_0, m_0)\in C^k(\R^n)\times \cM_{+}(\R^n)$, there exists a unique solution to \eqref{e:mmeas} on the time interval $[0,T)$. Moreover if \eqref{e:Oe} holds for some $\e$ on that interval then the solution can be extended beyond time $T$.
\end{THEOREM}
A similar bound from below on $|\det \n \bX(\cdot,t)|$ follows classically from the Liouville equation for $\det \n \bX(\cdot,t)$ and can be stated directly in Eulerian terms:
\begin{equation}
\label{e:contcrit}
\int_0^{T} \inf_{x\in \R^n} \n \cdot \bu(x,t)\dt > -\infty.
\end{equation}	
Since the initially non-negative $e_0$ remains bounded, and $\phi \ast \rho$ is always bounded, this implies \eqref{e:contcrit}. Consequently, we obtain global existence as shown in \cite{LearShvydkoy2019}.  For our purposes such approach is not productive, however, as we seek to obtain quantitative bi-Lipschitz bounds on the flow map, as in \eqref{e:Oe}, to extract further properties of the limiting mass-measure.

\subsection{Global Wellposedness and Proof of \eqref{e:Xbarsep}}\label{ss:gwp}

For unidirectional solutions we have 
\[
\bV(\a,t) = (V(\a,t),0),
\qquad 
\bX(\a,t) = (X(\a,t), \a_-), 
\qquad 
\a = (\a_1,\a_-).
\]
Then \eqref{e:Lag} becomes a scalar system in terms of the active flow components only:
\begin{equation}
\label{e:Laguni}
\left\{ 
\begin{split}
& \dot{X}(\a,t) = V(\a,t),  \\
& \dot{V}(\a,t) = -\k \int_{\O} \phi(X(\a,t)-X(\g,t),\;\a_- - \g_-))[V(\a,t) - V(\g,t)]\dm_0(\g), \\
& X(\a,0) = \a_1, \;V(\a,0) = u_0(\a).
\end{split}
\right.
\end{equation} 
The continuation criterion \eqref{e:Oe} takes form
\begin{equation}
\label{e:Oeuni}
\inf_{\a\in \O} |\p_{\a_1}X(\a,t)|>\e,
\qquad t\in [0,T).
\end{equation}
Following \cite{L2019CTC}, one can reduce the system \eqref{e:Laguni} to a single equation for $X(\a,t)$:
\begin{equation}
\label{e:Xdot}
\dot{X}(\a,t)
= f_0(\a) - \k \int_{\R^n} \f(X(\a,t)-X(\g,t), \a_- - \g_-)\dm_0(\g),
\end{equation}
where 
\[
\f(x_1,x_-) = \int_0^{x_1} \phi(y, x_-)\dy, 
\qquad \qquad 
f_0(\a) = u_0(\a) + \k\int_{\R^n} \f(\a-\g)\dm_0(\g).
\]
(Note that the quantity $f_0$ is related to $e_0$ via $e_0 = \p_{\a_1} f_0$.) 

Consider equation \eqref{e:Xdot} along two trajectories originating on the same $\a_-$-slice, and take the difference:
\begin{equation}\label{e:Xsep}
\begin{split}
\dot{X}(\a_1+h,\a_-,t) - \dot{X}(\a_1,\a_-,t)
& = \int_{\a_1}^{\a_1+h} e_0(\zeta,\a_-)\dd\zeta \\
&- \k \int_{X(\a_1,\a_-,t)}^{X(\a_1+h,\a_-,t)} \int_{\R^n} \phi(y - X(\zeta,t), \a_- - \zeta_-)\dm_0(\zeta)\dy.
\end{split}
\end{equation}
We use
\begin{equation}
\label{e:convbd}
\uphi M_0 
\le \int_{\R^n} \phi(\a_1 - X(\zeta,t), \a_- - \zeta_-)\dm_0(\zeta)
\le \|\phi\|_{L^\infty} M_0,
\end{equation}
in order to turn \eqref{e:Xsep} into a differential inequality.  (The upper bound in \eqref{e:convbd} is valid for all $\a\in \R^n$; the lower bound is valid for $\a\in \O$.  We only require the upper bound for the purposes of global existence; the lower bound will be useful later.) We denote 
\[
r(t) = X(\a_1+h,\a_-,t) - X(\a_1,\a_-,t)
\]
in the following:
\begin{equation}
\label{e:Xsepineq}
\int_{\a_1}^{\a_1+h} e_0(\zeta,\a_-)\dd\zeta - \k M_0 \|\phi\|_{L^\infty} r(t)
\le  \dot{r}(t)
\le \int_{\a_1}^{\a_1+h} e_0(\zeta,\a_-)\dd\zeta - \k M_0 \uphi \ r(t).
\end{equation}
The proof of the bound \eqref{e:Xbarsep} is completed simply by integrating the differential inequality \eqref{e:Xsepineq} and taking $t\to +\infty$. It also shows \eqref{e:Oeuni} with $\e(T)=\exp(-\k M_0 \|\phi\|_{L^\infty} T)>0$, for all $T\ge 0$, provided $e_0\ge 0$, so that the solution exists for all time. 

\begin{THEOREM}
For any unidirectional initial data $(\bu_0, m_0)\in C^k(\R^n)\times \cM_+(\R^n)$, there exists a unique global-in-time solution to \eqref{e:mmeas} if and only if $e_0\ge 0$.  
\end{THEOREM}

\begin{REMARK}
The Lagrangian argument naturally offers more detailed information about the particle trajectories than the original approach of \cite{LearShvydkoy2019}. However, it should be noted that the Eulerian approach is stable to perturbations and extends global existence to `almost' unidirectional solutions.
\end{REMARK}

\subsection{Flocking Estimates}

\label{ss:flocking} 

We now establish flocking estimates for the Euler Alignment system.  This will facilitate our stability analysis in the following subsection; furthermore, the exponentially decaying bound on $\n \bV$ will allow us to streamline the proof of Proposition \ref{p:Xreg} below.  Note that in this subsection and the following one, we do \textit{not} assume unidirectionality, but we do assume a heavy-tailed kernel.

\subsubsection{Basic Flocking and Alignment Bounds}
As above, we consider a flock of finite diameter and a compact domain $\Omega$ containing $\supp m_0$.  We define the flock parameters as follows:
\[
\cD_{\Omega}(t)=\max_{\a,\b\in\Omega}|\bX(\a,t)-\bX(\b,t)|, \qquad  \cA_{\Omega}(t)=\max_{\a,\b\in\Omega}|\bV(\a,t)-\bV(\b,t)|. 
\]
Since the domain $\Omega$ is fixed for all time, one can mimic the standard flocking argument for the discrete Cucker--Smale system by applying the Rademacher Lemma. The result is
\[
\ddt \cA_{\Omega}(t)\leq -\k M_0\phi(\cD_{\Omega}(t))\cA_{\Omega}(t).
\]
In particular, for heavy-tailed kernels we obtain flocking and exponentially fast alignment:
\begin{equation}\label{e:flocking}
\sup_{t\geq 0}\cD_{\Omega}(t)\leq \overline{\cD}_{\Omega}, \qquad \cA_{\Omega}(t)\leq \cA_{\Omega}(0)e^{-\k M_0 \phi(\overline{\cD}_{\Omega}) t}.
\end{equation}

\subsubsection{Bounds on the Deformation Tensor}
The bound \eqref{e:flocking} has appeared previously in, for example, \cite{TT2014}.  We now provide a new refinement of this flocking behavior by establishing estimates on the deformation tensor of the flow map.  Differentiating \eqref{e:Lag}, we obtain the following, for all $t\geq 0$ and $\a\in\Omega$:
\begin{align}
\label{e:nX}
\n \dot\bX(\a,t) & = \n \bV(\a,t),\\ 
\vspace{1 mm}
\label{e:nV}
\begin{split}
\n \dot\bV(\a,t) & = -\k  \n^\top \bX(\a,t) \int_{\R^{n}} \n \phi(\bX(\a,t) - \bX(\gamma,t)) \otimes (\bV(\a,t) - \bV(\g,t) ) \dm_0(\g)  \\
& \qquad - \n  \bV(\a,t)  \int_{\R^{n}}  \phi(\bX(\a,t) - \bX(\g,t)) \dm_0(\g).
\end{split}
\end{align}
Here, $\n^\top \bX(\a,t)$ denotes the matrix transpose of $\n \bX(\a,t)$.   
Combining \eqref{e:flocking}--\eqref{e:nV}, we get
\begin{align*}
\ddt \|\n \bX\|_{L^{\infty}(\Omega)} 
& \leq \|\n \bV\|_{L^{\infty}(\Omega)},\\
\ddt \|\n \bV\|_{L^{\infty}(\Omega)} 
& \leq \k M_0 \|\n \phi\|_{\infty} \cA_{\Omega}(0)e^{-\k M_0 \phi(\overline{\cD}_{\Omega})t} \|\n \bX\|_{L^{\infty}(\Omega)}-\k M_0 \phi(\overline{\cD}_{\Omega})\|\n \bV \|_{L^{\infty}(\Omega)}.  
\end{align*}
Let us simply rewrite it as
\begin{equation}\label{e:system(x,v)}
\dot{x}\leq v, \qquad \dot{v} \leq a e^{-bt}x-bv.
\end{equation}
where $a= \k M_0 \|\n \phi\|_{\infty} \cA_{\Omega}(0), b=\k M_0 \phi(\overline{\cD}_{\Omega}).$ Indeed, denoting $w = ve^{bt}$ we obtain
\[
\dot{x}\leq w e^{-bt}, \qquad \dot{w}\leq ax.
\]
Multiplying by factors to equalize the right hand sides, we obtain
\[
\ddt(ax^2 +e^{-bt}w^2)\leq 4axw e^{-bt}\leq 2 e^{-bt/2}\sqrt{a}(ax^2+e^{-bt}w^2).
\]
This immediately implies
\begin{equation}\label{e:aux(x,v)}
ax^2 +e^{bt}v^2\leq \tfrac{4\sqrt{a}}{b}(ax_0^2+v_0^2).
\end{equation}
In addition, we can read off bounds for each parameter individually:
\begin{equation}\label{e:final(x,v)}
x\leq \tfrac{2}{a^{1/4} b^{1/2}}\sqrt{a x_0^2 + v_0^2}, \qquad v \leq e^{-\frac{bt}{2}}\tfrac{2 a^{1/4}}{b^{1/2}}\sqrt{a x_0^2 + v_0^2}.
\end{equation}
Noting that $\n \bX(\cdot, 0) = \text{Id}, \n V(\cdot, 0) = \n \bu_0$, the estimate \eqref{e:aux(x,v)} implies
\begin{equation}\label{e:final(X,V)}
a\|\n \bX\|_{L^{\infty}(\Omega)}^2+e^{bt}\|\n \bV\|_{L^{\infty}(\Omega)}^2\leq \tfrac{4\sqrt{a}}{b}\left(a+\|\n\bu_0\|_{L^{\infty}(\Omega)}^2\right).
\end{equation}

\subsection{Stability}\label{s:stab}
 
We now turn our attention to stability estimates.  
\subsubsection{The KR Distance}

We measure the distance between two mass measures $m_t'$ and $m_t''$ using the Wasserstein-1 metric $W_1$.  We assume these measures have equal mass $M_0$ and zero momentum, and have support inside the same convex, compact set $\O$.  By the Kantorovich-Rubinstein Theorem, the distance between two such measures $\mu$ and $\nu$ is 
\begin{equation}\label{def:W1}
W_1(\mu,\nu)=\sup_{\text{Lip}(f)\leq 1}\left|\int_{\R^n}f(\gamma)\mbox{d}\mu(\gamma)-\int_{\R^n}f(\gamma)\mbox{d}\nu(\gamma) \right|.
\end{equation}
Note that a sequence of such measures with $\supp\mu_n\subset\Omega$ satisfies $W_1(\mu_n,\mu)\rightarrow 0$ if and only if $\mu_n \overset{\ast}{\rightharpoonup} \mu$.

\subsubsection{Stability of the Flow Map}
Let us consider two solutions $m_t'$, $m_t''$ on a common time interval of existence $[0,T)$, and let $(\bX', \bV')$ and $(\bX'', \bV'')$ denote the associated flow maps. We also denote the flock parameters by $\cD_\O'$, $\cD_\O''$, $\cA_\O'$, $\cA_\O''$, and the initial velocities by $\bu_0'$, $\bu_0''$.  Clearly,
\[
\ddt \| \bX'-\bX''\|_{L^{\infty}(\O)}\leq \| \bV'- \bV''\|_{L^{\infty}(\O)}.
\]
For the velocities, note that $\|\bV' - \bV''\|_{L^\infty(\O)}$ is a Lipschitz function in time; assume without loss of generality it is differentiable at time $t$.  Let $\ell\in (\R^n)^*$, $|\ell| = 1$ and $\a\in \O$ be a maximizing couple such that at time $t$ we have $\ell[\bV'(\a,t) - \bV''(\a,t)] = \|\bV' - \bV''\|_{L^\infty(\O)}$.  Then again by Rademacher's Lemma, we have
\begin{align*}
\ddt &\| \bV'-\bV''\|_{L^{\infty}(\O)}
= \ell (\dot\bV'(\a,t) - \dot\bV''(\a,t)) \\
& = \k\int_{\O}\phi(\bX'(\a,t)-\bX'(\g,t))\ell[\bV'(\g,t)-\bV'(\a,t)] \dm_0(\g)\\
& \quad -\k\int_{\O}\phi(\bX''(\a,t)-\bX''(\g,t))\ell[\bV''(\g,t)-\bV''(\a,t)]\dm_0''(\g)\\
& = \k\int_{\O} \phi(\bX'(\a,t)-\bX'(\g,t)) \ell[ \bV'(\g,t)-\bV'(\a,t)] [\dm_0'(\g)-\dm_0''(\g)]\\
& \quad + \k \int_{\O}\left[\phi(\bX'(\a,t)-\bX'(\g,t))-\phi(\bX''(\a,t)-\bX''(\g,t))\right]\ell[\bV'(\g,t)-\bV'(\a,t)]\dm_0''(\g)\\
& \quad + \k\int_{\O}\phi(\bX''(\a,t)-\bX''(\g,t))\ell\big[(\bV'(\g,t)-\bV''(\g,t))-(\bV'(\a,t)-\bV''(\a,t))\big]\dm_0''(\g).
\end{align*}
We label the terms on the right $I$, $II$ and $III$ and estimate them in turn. For $I$, we use the KR-distance:
\[
|I| \le 
\k \|\phi\|_{W^{1,\infty}}\left(\|\n \bX'\|_{L^{\infty}}\cA_{\O}'(t)+\|\n \bV'\|_{L^{\infty}} \right)W_1(m_0',m_0'').
\]
The second term is bounded by
\[
|II| \le 2\k M_0\|\n \phi\|_{\infty} \| \bX'- \bX''\|_{L^{\infty}(\O)}\cA_{\O}'(t).
\]
For the last term we use maximality of $\ell[\bV'(\a,t) - \bV''(\a,t)]$ and pull out the kernel first:
\begin{align*}
III
& = \k\int_{\O}\phi(\bX''(\a,t)-\bX''(\g,t))\ell\big[(\bV'(\g,t)-\bV''(\g,t))-(\bV'(\a,t)-\bV''(\a,t))\big]\dm_0''(\g) \\
&\leq \k \phi(\cD_{\Omega}''(t)) \int_{\O}\ell\big[(\bV'(\g,t)-\bV''(\g,t))-(\bV'(\a,t)-\bV''(\a,t))\big]\dm_0''(\g)\\
& = \k \phi(\cD_{\O}''(t)) \ell\left[\int_{\O}(\bV'(\g,t)-\bV''(\g,t))\dm_0''(\g)\right]-\k M_0 \phi(\cD_{\O}''(t))\| \bV' - \bV''\|_{L^{\infty}(\O)}\\
& = \k \phi(\cD_{\O}''(t)) \ell \int_{\O}\bV'(\g,t)[\dm_0''(\g) - \dm_0'(\g)] - \k M_0 \phi(\cD_{\O}''(t)) \| \bV' - \bV'' \|_{L^{\infty}(\O)}.
\end{align*}
In the last step we used (twice) equality of momenta: $\int \bV' \dm_0' = \int \bV'' \dm_0'' = 0$. Continuing,
\[
III \leq \k \|\phi\|_{\infty}\|\n \bV'\|_{L^{\infty}(\O)}W_1(m_0',m_0'')-\k M_0 \phi(\cD_{\O}''(t)) \| \bV' - \bV'' \|_{L^{\infty}(\O)}.
\]

Putting all the estimates together we obtain the system
\begin{align*}
\ddt \| \bX'-\bX''\|_{L^{\infty}(\O)}
& \leq  \| \bV'-\bV''\|_{L^{\infty}(\O)},\\
\ddt \| \bV'-\bV''\|_{L^{\infty}(\O)} 
& \leq 2\k \|\phi\|_{W^{1,\infty}}\left(\|\n \bX'\|_{L^{\infty}}\cA_{\O}'(t)+\|\n \bV'\|_{L^{\infty}} \right)W_1(m_0',m_0'') \\
& \quad + 2\k M_0\|\n \phi\|_{\infty} \| \bX'- \bX''\|_{L^{\infty}(\O)}\cA_{\O}'(t) \\
& \quad -\k M_0 \phi(\cD_{\O}''(t)) \| \bV' - \bV'' \|_{L^{\infty}(\O)}.
\end{align*}
Using the estimate \eqref{e:aux(x,v)} on the deformation tensor and \eqref{e:flocking} on the diameter
and amplitude, we conclude
\[
\ddt \| \bV'-\bV''\|_{L^{\infty}(\O)} \leq a e^{-b t}\left[W_1(m_0', m_0'') + \| \bX'-\bX''\|_{L^{\infty}(\O)} \right]-b \| \bV'-\bV''\|_{L^{\infty}(\O)}.
\]
So, we obtain the same system \eqref{e:system(x,v)} as in our flocking estimates, but for the new pair
\[
x=W_1(m_0',m_0'') + \| \bX'-\bX''\|_{L^{\infty}(\O)}, \qquad v=\| \bV'-\bV''\|_{L^{\infty}(\O)}.
\]
Using \eqref{e:final(x,v)} and recalling that our initial quantities are now 
\[
x(0)=W_1(m_0', m_0'')\quad  \mbox{and}   \quad  v(0)=\|\bu_0'-\bu_0''\|_{L^{\infty}(\Omega)},
\]
 we obtain the following for kernels with heavy tail
\begin{align}
\label{e:stability1}
\| \bX'-\bX''\|_{L^{\infty}(\O)} & \leq C \left[ W_1(m_0', m_0'') + \|\bu_0'-\bu_0''\|_{L^{\infty}(\O)}\right],
\\
\label{e:stability2}
\| \bV'-\bV''\|_{L^{\infty}(\O)}
& \leq C e^{-ct} \left[ W_1(m_0', m_0'') + \|\bu_0'-\bu_0''\|_{L^{\infty}(\O)}\right].
\end{align}
The above inequalities hold for all time $t\in [0,T)$, and $C, c > 0$ depend only on the initial diameters of the flocks, the common mass $M_0$, and the kernel $\phi$.

\subsubsection{Stability of the Mass Measure}
The estimates \eqref{e:stability1} and \eqref{e:stability2} already express stability of the characteristics of the flock; however, the ultimate application lies in estimating the KR-distance $W_1(m_t',m_t'')$ and establishing contractivity of the dynamics. Toward this end, let us fix a function $f$ with $\text{Lip}(f) \leq 1$, and write
\begin{align*}
& \int_{\O} f(\g) \dm_t'(\g) - \int_{\O} f(\g)\dm_t''(\g) = \int_{\O} f (\bX'(\g,t))\dm_0'(\g) - \int_{\O} f (\bX''(\g,t))\dm_0''(\g)\\
& = \int_{\O} f (\bX'(\g,t))[\dm_0'(\g)-\dm_0''(\g)]-\int_{\O} [f (\bX'(\g,t))-f(\bX''(\g,t))] \dm_0''(\g) \\
&\leq \text{Lip}(f(\bX)) W_1(m_0', m_0'') + M_0 \|\bX'-\bX''\|_{L^{\infty}(\O)}.
\end{align*}
Using $\text{Lip}(f(\bX))\leq \|\n \bX(t)\|_{L^{\infty}}$ and applying the deformation and stability estimates \eqref{e:final(X,V)}, \eqref{e:stability1}, we get
\[
W_1(m_t', m_t'')\leq  C \left[ W_1(m_0', m_0'')+\|\bu_0'-\bu_0''\|_{L^{\infty}(\O)}\right].
\]
Since this estimate holds for all time, passing to the limit $t \to \infty$ we make the same conclusion for the limiting measures $\barm':=\barbX'_\sharp m_0'$ and $\barm'':=\barbX''_\sharp m_0''$:
\[
W_1(\barm', \barm'')\leq  C \left[ W_1(m_0', m_0'')+\|\bu_0'-\bu_0''\|_{L^{\infty}(\O)}\right].
\]

\section{Concentration of Mass for Unidirectional Solutions}

\subsection{Horizontal Slices of $\barbX$ and the Lebesgue Decomposition of $\barm$}

\label{s:barbX}

In this section, we use \eqref{e:Xbarsep} to establish the remaining properties of the limiting flow map that comprise the statement of Proposition \ref{p:barbX}.  Then we prove Theorem \ref{t:main} as a consequence.

\subsubsection{Consequences of \eqref{e:Xbarsep}}

We use the following notation:
\[
\begin{split}
\O_{\a_-} & = \{\a_1\in \R: (\a_1, \a_-)\in \O\},
\quad \quad 
\cZ_{\a_-} = \{\a_1\in \R: (\a_1, \a_-)\in \cZ\}, \\
\cP_{\a_-} & = \{\a_1\in \R: (\a_1, \a_-)\in \cP\} = \R\backslash \cZ_{\a_-},\quad 
\barX_{\a_-}(\a_1) = \barX(\a_1,\a_-).
\end{split}
\]

The statements in the following Corollary must be collected, but their proofs are trivial using  \eqref{e:Xbarsep}.
\begin{COROL} For each $\a_-\in \R^{n-1}$, the following statements are true:
\begin{itemize}
	\item The map $\barX_{\a_-}$ is monotonically increasing, with $\barX_{\a_-}(\b) = \barX_{\a_-}(\g)$ if and only if $\int_{\b}^{\g} e_0(\zeta, \a_-)\dd\zeta~=~0$.  
	\item The map $\a_1\mapsto \barX_{\a_-}(\a_1)$ is absolutely continuous, therefore a.e. differentiable. Furthermore, we have the following upper and lower bounds, valid for $\a=(\a_1, \a_-)\in \O\backslash \p\cZ$:
	\begin{equation}
	\label{e:pa1Xbarbds}
	\frac{e_0(\a_1, \a_-)}{\k M_0 \|\phi\|_{L^\infty}}
	\le \p_{\a_1} \barX_{\a_-}(\a_1) 
	\le \frac{e_0(\a_1, \a_-)}{\k M_0 \uphi}.
	\end{equation}
\end{itemize}
\end{COROL}

Next, we demonstrate that the bound \eqref{e:Xbarsep} can be used to estimate the effect of $\barX$ on the Lebesgue measure of a set.
\begin{PROP}
	\label{p:Xm}
	Let $E$ be a bounded, measurable subset of $\mathbb{R}$.  Then 
	\begin{equation}
	\label{e:XbarE}
	\frac{1}{\k M_0 \|\phi\|_{L^\infty}}\int_{E} e_0(\a_1, \a_-) \dalpha_1 \le |\barX_{\a_-}(E)|\le \frac{1}{\k M_0 \uphi}\int_{E} e_0(\a_1, \a_-)\dalpha_1.
	\end{equation}	
	The upper bound requires the additional assumption that $E\subset \O_{\a_-}$.  In particular,
	\begin{equation}
	\label{e:XZm0}
	|\barX_{\a_-}(\cZ_{\a_-})| = 0.
	\end{equation} 
\end{PROP}
\begin{proof}
	It suffices to prove the bounds for open sets $E$, using outer regularity to extend them to all bounded measurable sets.  We prove the upper bound first.  Writing $E$ as a countable union of disjoint open intervals $(\b_i, \g_i)$, we have from \eqref{e:Xbarsep} that 
	\[
	|\barX_{\a_-}(E)| \le \sum_{i=1}^\infty |\barX_{\a_-}(\b_i) - \barX_{\a_-}(\g_i)| \le \frac{1}{\k M_0 \uphi} \int_{E} e_0(\a_1, \a_-)\dalpha_1. 
	\]	
	This proves the upper bound, and \eqref{e:XZm0} follows.  Next, write $E\backslash \cZ_{\a_-}$ as a countable union of disjoint open intervals $(\widetilde{\b}_i, \widetilde{\g}_i)$.  Then using \eqref{e:XZm0}, \eqref{e:Xbarsep}, and the fact that $\barX_{\a_-}$ is strictly increasing on $\cP_{\a_-}$ (and therefore maps disjoint open intervals in $\cP_{\a_-}$ to disjoint open intervals in $\R$), we get 
	\[
	|\barX_{\a_-}(E)| 
	= |\barX_{\a_-}(E\backslash \cZ_{\a_-})|
	= \sum_{i=1}^\infty |\barX_{\a_-}(\widetilde{\b}_i) - \barX_{\a_-}(\widetilde{\g}_i)| 
	\ge \frac{1}{\k M_0 \|\phi\|_{L^\infty}} \int_{E} e_0(\a_1, \a_-)\dalpha_1,
	\]	
	which establishes the lower bound.
\end{proof}

Integrating the inequalities \eqref{e:XbarE} over $\R^{n-1}$ yields the following Corollary, which completes the proof of Proposition \ref{p:barbX}.
\begin{COROL}
	\label{c:XbarZ}
	Let $E$ be a bounded, measurable subset of $\R^n$.  Then 
	\begin{equation}
	\label{e:XbarE2}
	\frac{1}{\k M_0 \|\phi\|_{L^\infty}}\int_{E} e_0(\a) \dalpha \le |\barbX(E)|\le \frac{1}{\k M_0 \uphi}\int_{E} e_0(\a)\dalpha.
	\end{equation}	
	The upper bound requires the additional assumption that $E\subset \O$.  In particular,
	\begin{equation}
	|\barbX(\cZ)|=0.
	\end{equation}
\end{COROL}

\subsubsection{Proof of Theorem \ref{t:main}}

The results of the previous subsection allow us to establish Theorem \ref{t:main}.

\begin{proof}[Proof of Theorem \ref{t:main}]
	
Let $E$ be a Lebesgue null set.  Then $E\cap \barbX(\cP) = \barbX(\barbX^{-1}(E)\cap \cP)$ is also Lebesgue null, whence $\int_{\barbX^{-1}(E)\cap \cP} e_0(\a)\dalpha = 0$ by the lower bound in Corollary \ref{c:XbarZ}.  Since $e_0>0$ on $\barbX^{-1}(E)\cap \cP$, this implies that $\barbX^{-1}(E)\cap \cP$ is Lebesgue null.  Thus,
	\[
\barbX_\sharp (\rho_0 \one_{\cP} \dx)(E) = \int_{\barbX^{-1}(E)\cap \cP} \rho_0 \dx = 0.
	\]
	This proves that $\barbX_\sharp (\rho_0 \one_{\cP} \dx) \ll \dx$. Next, we note that the support of $\barbX_\sharp (\rho_0 \one_{\cZ} \dx + \dnu)$ is contained in $\barbX(\cZ\cup\supp\nu)$, which is Lebesgue null by Corollary \ref{c:XbarZ}. This proves that $\barbX_\sharp (\rho_0 \one_{\cZ} \dx + \dnu)\perp \dx$.

	The formula for $\barrho\circ \barbX$ in \eqref{e:barmdecomp2} follows simply from the fact that $\barrho \dx$ is the pushforward of $\rho_0 \one_\cP\dx$ under $\barbX$ (and $\det\n \barbX = \p_{\a_1} \barX$).  Similarly, we have $\rho(\bX(\cdot,t), t) = \frac{\rho_0}{\p_{\a_1} X }$.  
	
	Peeking ahead to (the easy part of) Proposition \ref{p:gradbarX}, we see that $\p_{\a_1} X\to \p_{\a_1} \barX$ uniformly as $t\to +\infty$; since $\p_{\a_1} \barX \ge c>0$ on compact subsets of $\cP$, it follows that $\rho(\bX(\cdot,t), t)\to \barrho\circ \barbX$ uniformly on compact subsets of $\cP$.  This completes the proof.
\end{proof}

\newpage
\subsection{Regularity of $\barbX$ and the Concentration Set}

\label{s:Xreg}

\subsubsection{Regularity of $\barbX$}

We have already seen that $\barbX$ is a continuous function, being the uniform limit of the maps $\bX(\cdot,t)$ as $t\to +\infty$.  We have also used \eqref{e:Xbarsep} to study the regularity of $\barbX$ in the $x_1$ direction, but we have not proved anything about the other directions.  We rectify this situation presently and prove that $\barbX$ is $C^1$ off of $\p \cZ$.  The two bounds \eqref{e:fastalignment} (established in \cite{LearShvydkoy2019}) and \eqref{e:final(X,V)} (established in Section \ref{ss:flocking} above) make this relatively straightforward.  

\begin{PROP}
\label{p:gradbarX}
The map $\barbX$ is continuously differentiable on $\R^n\backslash \p\cZ$, and $\n\bX(t)$ converges uniformly on compact subsets of $\R^n\backslash \p\cZ$ as $t\to +\infty$.
\end{PROP}
The statement above of course implies that the limit is $\n \barbX$ away from $\R^n\backslash \p \cZ$.

\begin{proof}
Taking a spatial derivative of the equation
\[
\dot{X}(\a,t) = u(X(\a,t), \a_-, t)
\] 
yields
\begin{align}
\label{e:p1ptX}
\p_{\a_1}\dot{X}(\a,t) 
& = \p_{x_1} u(X(\a,t),\a_-,t) \p_{\a_1}X(\a,t); \\
\label{e:pjptX}
\p_{\a_j}\dot{X}(\a,t) & = \p_{x_1} u(X(\a,t),\a_-,t) \p_{\a_j}X(\a,t) + \p_{x_j} u(X(\a,t),\a_-,t), \quad j \ne 1.
\end{align}
Combining \eqref{e:p1ptX} and \eqref{e:pjptX} with  the exponential decay of $\n u$ along trajectories originating in $\cP_\e\cap \O$, we conclude that $\n_\a \bX(t) \to \n_\a \barbX$ uniformly on any $\cP_\e\cap \O$ and thus that $\barbX$ is $C^1$ on $\cP$.  
	
We now show that $\n_\a \bX(t)$ converges uniformly on $\cZ\cap \O$, which will guarantee that $\barbX$ is continuously differentiable in the interior of  $\cZ$.  This is slightly harder than working inside $\cP_\e\cap \O$, since we no longer have the bound \eqref{e:fastalignment}.  Instead, we take advantage of the fact that 
\begin{equation}
\label{e:e0}
e(\bX(\a,t), t) = 0,\quad \quad \a\in \cZ,
\end{equation}
and 
\begin{equation}
\label{e:px1uZ}
\p_{x_1} u(X(\a,t), \a_-, t) = -\k \phi*m_t(\bX(\a,t)) \le -\k M_0 \uphi, 
\quad \quad \a\in \cZ\cap \O.
\end{equation}

Inserting \eqref{e:px1uZ} into \eqref{e:p1ptX} already shows that $\p_{\a_1} X(\a,t)\to 0$ uniformly on $\cZ\cap \O$, and so we recover the fact that $\p_{\a_1} \barX(\a) \equiv 0$ in the interior of $\cZ$ (which we already knew).  

Now assume $j\ne 1$.  We proceed using the identity
\begin{equation}
\p_{\a_j} X(\a,t)
= \frac{\p_{x_j} u(\bX(\a,t),t) - \p_{\a_j} V(\a,t)}{\k \phi*m_t(\bX(\a,t))}, \quad \quad \a\in \cZ.	
\end{equation}

Since we already know by \eqref{e:final(X,V)} that $\p_{\a_j} V\to 0$ uniformly on $\O$, and that $\phi*m_t(\bX(\cdot, t), t)$ is bounded away from zero on $\cZ\cap \O$, it suffices to show that $\p_{x_j} u(\bX(\cdot,t),t)$ and $\phi*m_t(\bX(\cdot,t))$ converge uniformly on $\cZ\cap 
\O$ as $t\to +\infty$.  The second of these points is clear.  Indeed,
\[
\phi*m_t(\bX(\a,t)) = \int \phi(\bX(\a,t) - \bX(\g,t))\dm_0(\g) \to \int \phi(\barbX(\a) - \barbX(\g))\dm_0(\g),
\]
uniformly in $\a\in \O$, by the uniform convergence of $\bX$ to $\barbX$ and the continuity of $\phi$.

As for the term $\p_{x_j}u(\bX(\a,t),t)$, we write
\begin{align*}
\frac{\dd}{\dt} \p_{x_j} u(\bX(\a,t), t)& = \int_{\R^n} \p_{x_j} \phi(\bX(\a,t) - y)(u(y,t) - u(\bX(\a,t),t))\dm_t(y) - e \p_{x_j} u(\bX(\a,t), t).
\end{align*}
Using \eqref{e:e0}, we get 
\begin{equation}
\label{e:pxjuZ}
\left| \frac{\dd}{\dt} \p_{x_j} u(\bX(\a,t), t) \right| 
\le \|\phi_{x_j}\|_{L^\infty} M_0 \cA(t) \le Ce^{-\d_\O t}, \quad \quad \a\in \cZ\cap \O.
\end{equation}
We may thus conclude that the function $\p_{x_j} u(\bX(\a,t),t)$ converges uniformly on $\cZ\cap \O$, as $t\to +\infty$.  This completes the proof.
\end{proof}

\subsubsection{Proof of Theorem \ref{t:Xreg}}

We now prove Theorem \ref{t:Xreg}.  The heavy lifting has been done already; we just need to put together the relevant statements.

\begin{proof}
Let $f$ and $U$ be as in the statement of Theorem \ref{t:Xreg}.  Denote $U_{\a_-} = \{\a_1\in U:(\a_1, \a_-)\in U\}$.  Then since $U\subset \cZ$, we have 
\[
\barX_{\a_-}(U_{\a_-}) = f(\a_-),
\]
by \eqref{e:Xbarsep}.  Thus 
\[
\barbX(U) 
= \{(\barX_{\a_-}(U_{\a_-}),\a_-): \a_-\in U_-\}
= \{(\barX_{\a_-}(f(\a_-)),\a_-): \a_-\in U_-\}.
\]
Since $\barbX$ is $C^1$ in the interior of $\cZ$ and $\a_-\mapsto (f(\a_-),\a_-)$ takes values in $U$, it follows that the function $\a_-\mapsto \barX(f(\a_-), \a_-)$ is $C^1$, so that $\barbX(U)$ is the graph of a $C^1$ function.

Next, assume that $\overline{U} = \cZ$, as in the second part of Theorem \ref{t:Xreg}.  Denote 
\[
\dbmu_{\cZ}^{x_-}(x_1) = (\barX_{x_-})_\sharp (\one_{\cZ}\rho_0(x_1, x_-)\dx_1).  
\]
Then we have 
\[
\dbmu_\cZ(x) = \dbmu_\cZ^{x_-}(x_1) \dx_-,
\]
just by unpacking the notation; we claim that in fact 
\[
\dbmu_\cZ^{x_-}(x_1) = c(\a_-)\d_{f(x_-)}(x_1),
\quad \quad 
c(x_-) = \int_{\cZ_{x_-}} \rho_0(x_1, x_-)\dx_1.  
\] 
Indeed, $\supp \overline{\mu}_{\cZ}^{x_-} \subset \barX_{x_-}(\cZ_{x_-}) = \{f(x_-)\}$, and $\overline{\mu}_{\cZ}^{x_-}(\{f(x_-)\}) = \int_{\barX_{x_-}^{-1}(\{f(x_-)\})} \rho_0(x_1,x_-)\dx_1 = c(x_-)$, which proves the claim.  The final statement of the Theorem, on the regularity of $c(x_-)$, is clear.  
\end{proof}

\section{Fine Properties of $\barm$ and $\barX$ in Dimension $1$}

\label{s:fine}

In this section we restrict attention to the case of a single space dimension, $n=1$.  Our main goal in this section is the proof of Theorem \ref{t:fineprops}, but more generally, we seek to demonstrate how we can tune $e_0$ to manipulate the fine properties of the limiting measure and flow map.

\subsection{Tuning the Dimension of $\barX(\cZ)$}

\label{ss:XZdim}

In this subsection, we construct an $e_0$ whose zero set $\cZ$ is a Cantor-type set of positive measure, such that $\barX(\cZ)$) has a specified dimension in $(0,1)$.  The construction will prove the sharpness claim in Theorem \ref{t:fineprops}.  We wait until the next subsection to spell out this connection explicitly.

Let $g:\R\to \R$ be smooth, nonnegative function, with $\{x:g(x)>0\} = (-\frac12, \frac12)$.  Choose $\g\in (0,\frac12)$ and $\b\in (0,1)$.  We start with the interval $[0,1]$ and remove the open center interval $J_1^1$ of length $\g$.  Call the remaining (closed) intervals $I_1^1$ and $I_1^2$.  Then remove the middle open intervals of length $\g^2$ from each of $I_1^1$ and $I_1^2$.  Call the removed intervals $J_2^1$ and $J_2^2$, respectively, and denote the remaining closed intervals $I_2^1, \ldots, I_2^4$.  We continue this process indefinitely.  For each $j\in \N$, $k\in \{1, \ldots, 2^{j-1}\}$, let $c_{j,k}$ denote the center of the interval $J_j^k$.  

We set 
\begin{equation}
\label{e:e0Cantor}
e_0(\a) = \sum_{j=1}^\infty \sum_{k=1}^{2^{j-1}} \b^j g\left( \frac{\a - c_{j,k}}{\g^j} \right), 
\quad  \quad \a\in [0,1],
\end{equation}
with $e_0>0$ on $\R\backslash [0,1]$, so that 
\begin{equation}
\label{e:defZCantor}
\cZ = \bigcap_{j=1}^\infty \cZ_j, 
\quad \quad \cZ_j:=\bigcup_{k=1}^{2^j} I_j^k.
\end{equation}
That is, $\cZ$ is a standard Cantor-like set of measure $1 - \g - 2\g^2 - 4\g^3 - \cdots = \frac{1-3\g}{1-2\g}$.  In particular, the Hausdorff dimension of $\cZ$ is $1$.  On the other hand, the dimension of the image $\barX(\cZ)$ depends on $\b$ and $\g$, according to the following Proposition.

\begin{PROP}
\label{p:Cantor}	
With $\cZ$ as defined in \eqref{e:defZCantor}, the set $\barX(\cZ)$ has Hausdorff dimension and box counting dimension equal to $\frac{\ln 2}{-\ln(\b\g)}$:
\[
\dim_\cH(\barX(\cZ)) = \dim_{\mathrm{box}}(\barX(\cZ)) = \frac{\ln 2}{-\ln(\b\g)}.
\]	
\end{PROP}

Note that by adjusting $\b\in (0,1)$ and $\g\in (0,\frac12)$, we can obtain any dimension between $0$ and $1$. 

We refrain from recalling the standard definitions of the Hausdorff and box-counting dimensions, but we remind the reader that the relationship between the Hausdorff dimension and box-counting dimension is summarized in the following inequality:
\begin{equation}
\label{e:dimensions}
\dim_\cH(E) \le \underline{\dim}_{\text{box}}(E) \le \overline{\dim}_{\text{box}}(E).
\end{equation}
The quantities in this inequality denote the Hausdorff dimension, lower box-counting dimension, and upper box-counting dimension.  If the upper- and lower- box-counting dimensions agree, their common value is the box-counting dimension (without qualifiers), denoted $\dim_{\text{box}}(E)$. Thus, in order to prove the Proposition, it suffices to prove an upper bound on the upper box-counting dimension, and a lower bound on the Hausdorff dimension.  We prove the former `by hand', but for the latter, we make use the following special case of Frostman's Lemma. (For a more general statement, see for example \cite{Mattila}.)

\begin{LEMMA}[Frostman]
	Let $E$ be a Borel subset of $\R$. Suppose there exists a Borel measure $\mu$ satisfying the following two conditions:
	\begin{itemize}
		\item There exist constants $c>0$ and $s\in [0,1]$ such that for all $x\in \R$ and $r>0$, the bound $\mu((x-r,x+r))<cr^s$ holds.
		\item $\mu(E)>0$.
	\end{itemize}   
	Then $\dim_\cH(E)\ge s$.
\end{LEMMA}	
\begin{proof}
	Choose $\d>0$ and cover $E$ by countably many intervals $I_i= (x_i-r_i, x_i+r_i)$, with $r_i<\d$. Then 
	\[
	0<\mu(E) \le \sum \mu(I_i) \le \sum cr_i^s 
	\] 		
	This shows that $\mu(E)\lesssim \cH^s_\d(E)$, for all $\d>0$, from which the conclusion follows.
\end{proof}

\begin{proof}[Proof of Proposition \ref{p:Cantor}]

\textbf{Step 1}: \textit{Upper Bound on the Box-Counting Dimension}. 
Denote $c_0 = \int g(x) \dx$.  For each $j,k$, we have
\begin{equation}
\int_{I_j^k} e_0(\a)\dalpha 
= \sum_{\ell={j+1}}^\infty 2^{\ell-(j+1)}(\b\g)^\ell c_0 = \frac{c_0}{2} \cdot \frac{(\b\g)^{j+1}}{1-2\b\g} =: c_1 (\b\g)^j.
\end{equation}
According to \eqref{e:Xbarsep}, it follows that the length of $\barX(I_j^k)$ is bounded above and below as follows.
\begin{equation}
\label{e:XIjk}
\frac{c_1}{\k M_0 \|\phi\|_{L^\infty}} (\b\g)^j 
\le |\barX(I_j^k)|
\le \frac{c_1}{\k M_0 \uphi} (\b\g)^j.
\end{equation}

Define $c_2:= \frac{c_1}{\k M_0 \uphi}$.  Choose $r>0$ small, and then choose $j\in \N$ so that 
\begin{equation}
\label{e:boxr}
c_2(\b\g)^{j} \le 2r < c_2(\b\g)^{j-1}. 
\end{equation}

For $R>0$ and $E\subset \R$, let $N(R;E)$ denote the minimal number of open intervals of radius $R$ required to cover the set $E$.  By construction, the set $\cZ$ is covered by the $2^j$ intervals $(I_j^k)_{k=1}^{2^j}$, so $\barX(\cZ)$ is covered by the $2^j$ intervals $(\barX(I_j^k))_{k=1}^{2^j}$, each of which has length at most $c_2(\b\g)^j<2r$.  Thus
\begin{equation}
\label{e:Nrupper}
N(r;\;\barX(\cZ)) \le 2^j,
\end{equation}
so that 
\[
\frac{\ln N(r,\;\barX(\cZ))}{-\ln (2r)} \le \frac{\ln(2^j)}{-\ln(c_2(\b\g)^{j-1})}
\]
Taking $r\to 0$ (and thus $j\to +\infty$) gives the desired upper bound on the upper box-counting dimension:
\begin{equation}
\label{e:boxupper}
\overline{\dim}_{\text{box}}(\barX(\cZ)) \le  \frac{\ln 2}{-\ln(\b\g)}.
\end{equation}

\textbf{Step 2}: \textit{Lower Bound on the Hausdorff Dimension}. We verify the hypotheses of Frostman's Lemma, fixing the following parameters: $\mu = \overline{\mu}_\cZ$, $s = \frac{\ln 2}{-\ln(\b\g)}$, $c = 4c_3^{-s}$, with $c_3:= \frac{c_1}{\k M_0 \|\phi\|_{L^\infty}}$ and $c_1$ as above.   

Choose $x\in \R$, $r>0$ small (without loss of generality), and put $I_* = (x-r,x+r)$.  Choose $j$ so that 
\[
c_3 (\b\g)^{j+1} \le 2r < c_3 (\b\g)^j;
\]
note that this implies (by \eqref{e:XIjk} and the definition of $c_3$) that $|I_*|=2r<|\barX(I_j^k)|$, for each $k\in \{1, \ldots, 2^{j}\}$.  Thus $I_*$ may intersect at most $2$ of the intervals $\barX(I_j^k)$, say $\barX(I_j^{k_1})$ and $\barX(I_j^{k_2})$.  Recalling that $\rho_0\equiv 1$ on $\cZ$, we obtain
\begin{equation}
\label{e:muZI*}
\overline{\mu}_\cZ(I_*) = m_0(\barX^{-1}(I_*)\cap \cZ) \le m_0(I_j^{k_1})+m_0(I_j^{k_2}) = |I_j^{k_1}| + |I_j^{k_2}| < 2^{1-j}.
\end{equation}

Since $(\b\g)^{-s} = 2$, an elementary calculation yields
\begin{equation}
\label{e:I*lwr}
|I_*|^s = (2r)^s \ge (c_3(\b\g)^{j+1})^s = c_3^s 2^{-j-1}.
\end{equation}
Combining \eqref{e:muZI*} and \eqref{e:I*lwr} yields
\[
\overline{\mu}_\cZ(I_*) \le \frac{4}{c_3^s} |I_*|^s,
\]
which shows that the hypotheses of Frostman's Lemma are satisfied and thus that 
\begin{equation}
\label{e:Hausdorfflwr}
\dim_{\cH}(\barX(\cZ))\ge s = \frac{\ln 2}{-\ln(\b\g)}.
\end{equation}

Combining \eqref{e:boxupper} and \eqref{e:Hausdorfflwr} completes the proof of the Proposition.
\end{proof}

\subsection{Regularity of $e_0$ and the Dimension of $\barX(\cZ)$}

In the previous subsection, we allowed ourselves to adjust both $\b$ and $\g$ in order to get the conclusion that any box-counting dimension is possible.  However, by adjusting $\b$ only, we can of course still obtain any dimension between $0$ and $-\ln(2)/\ln(\g)$.  Since $\cZ$ depends only on $\g$ and not on $\b$, this already demonstrates that the dimension of $\barX(\cZ)$ depends not only on $\cZ$ itself, but on the way $e_0$ approaches zero near $\cZ$, as encoded in the parameter $\b$. In fact, note that $e_0\in C^k$ if and only if $\b\le \g^k$, and the latter implies 
\[
\dim_{\text{box}}(\barX(\cZ)) = \frac{\ln 2}{-\ln(\b\g)} \le \frac{\ln 2}{-(k+1)\ln(\g)}<\frac{1}{k+1},
\quad \quad \text{ if } e_0\in C^k(\R).
\]
This motivates the statement of Theorem \ref{t:fineprops}, whose proof we give presently.  

\begin{proof}[Proof of Theorem \ref{t:fineprops}]

We have essentially already proved the sharpness part, since if $\b =\g^k$ in Proposition \ref{p:Cantor}, we see that 
\[
\dim_{\text{box}}(\barX(\cZ)) = \frac{1}{k+1}\cdot \frac{\ln 2}{\ln(1/\g)},
\]
and the right side can be made arbitrarily close to $\frac{1}{k+1}$ if $\g$ is sufficiently close to $\frac12$.  

Next, we prove \eqref{e:boxdime0} in the case where $e_0\in C^k(\R)$; the statement for $e_0\in C^\infty$ follows immediately. Assume without loss of generality that $\cZ$ is a perfect set (i.e., contains no isolated points).  Then $e_0$ and its first $k$ derivatives vanish at any point of $\cZ$, so that the Taylor expansion of $e_0$ gives 
\begin{equation}
e_0(\a) \le C(\dist{\a}{\cZ})^k.
\end{equation}
Now, for any $r>0$, we can cover $\cZ$ by $2N(r;\,\cZ)$ intervals of radius $r$ centered at points $(x_i)_{i=1}^{2N(r;\,\cZ)}$ in $\cZ$.  Using \eqref{e:Xbarsep}, we have 
\begin{equation}
|\barX(B(x_i,r))| \lesssim \int_{B(x_i,r)} e_0(\a)\dalpha \le 2C\int_0^r \a^k \dalpha  \lesssim r^{k+1}.
\end{equation}
Thus, there exists $C>0$ such that 
\[
N(Cr^{k+1}; \barX(\cZ)) \le 2N(r;\,\cZ), 
\quad \quad r>0.
\]
Thus 
\[
\frac{\ln(N(Cr^{k+1};\;\barX(\cZ)))}{-\ln(Cr^{k+1})} \le \frac{\ln N(r;\,\cZ)}{-(k+1)\ln r - \ln C}.
\]
Taking $r\to 0^+$ yields the desired conclusion.
\end{proof}

\subsection{Local Dimension of $\barm$}

We argued above that the dimension of $\barX(\cZ)$ depends on both $\cZ$ and the rate at which $e_0$ approaches zero near $\cZ$; we used smoothness of $e_0$ to control the latter.  We now demonstrate that something similar is true for the local dimension of $\barm$, using a simpler construction. The following Proposition gives an example of how to tune $e_0$ to obtain a given local dimension of $\barm$ at a specified point.  The Proposition is stated for an isolated point of $\cZ$, near which $\rho_0$ is constant and $e_0$ is a power-law function. 

\begin{PROP}
	Assume that $\rho_0$ and $e_0$ are both even functions, and that (as usual) $\overline{u}_0 = 0$.  Let $p$ be any real number greater than $1$, and assume that for some $\d>0$, we have 
	\[
	\rho_0(\a) = 1, \quad \quad e_0(\a) = p|\a|^{p-1},
	\quad \quad |\a|<\d.
	\]
	Then the local dimension $d(x,\barm)$ of $\barm$ at $x = 0$ is
	\[
	d(0, \barm) := \lim_{r\to 0} \frac{\ln (\barm(-r,r))}{\ln r} = \frac1p.  
	\]
\end{PROP}

Since $p>1$ is arbitrary, it follows that any local dimension in $[0,1]$ can be attained.  (The cases $d(0,\barm) = 0, 1$ are trivial.)

\begin{proof}
	Note first of all that the hypotheses guarantee that $\barX$ is an odd function.  Choose $r>0$ small, and then choose $s>0$ such that $\barX(s) = r$.  Then 
	\[
	\barm(-r,r) = m(-s,s) = 2s.
	\]
	On the other hand, if $r$ is small enough so that $s<\d$, then by \eqref{e:Xbarsep}, it follows that   
	\begin{equation}
	\frac{s^p}{\|\phi\|_{L^\infty}} = \frac{1}{\|\phi\|_{L^\infty}} \int_0^s e_0(\a)\dalpha \le \k M_0 r \le \frac{1}{\uphi} \int_0^s e_0(\a)\dalpha = \frac{s^p}{\uphi}.
	\end{equation}
	Thus 
	\[
	\frac{\ln 2s}{\ln Cs^p}
	\le \frac{\ln (\barm(-r,r))}{\ln r}
	\le  \frac{\ln 2s}{\ln cs^p} 
	\]
	Taking $r\to 0^+$ yields the desired statement.  
\end{proof}

%\bibliographystyle{plain}
%\bibliography{bib12020}

\def\cprime{$'$}

\end{document}